\documentclass[11pt]{amsart}
\usepackage{amsmath}
\usepackage{stmaryrd}
\usepackage{epsfig,color}
\usepackage{blindtext}
\usepackage{enumerate}
\usepackage{enumitem}
\usepackage{hyperref}
\usepackage{graphicx}
\usepackage{url}
\usepackage{bbm}
\usepackage{amssymb}
\usepackage{filecontents}
\usepackage{nicefrac,mathtools}
\usepackage{verbatim} 

\newtheorem{theorem}{Theorem}[section]
\newtheorem{definition}[theorem]{Definition}

\newtheorem{proposition}[theorem]{Proposition}
\newtheorem{lemma}[theorem]{Lemma}
\newtheorem{remark}[theorem]{Remark}

\newtheorem{question}[theorem]{Question}

\newtheorem{claim}[]{Claim}
\newtheorem*{acknowledgements}{Acknowledgements}

\newtheorem*{remark*}{Remark}

\numberwithin{equation}{section}

\newcommand{\mf}{\mathbf}
\newcommand{\mc}{\mathcal}
\newcommand{\mb}{\mathbb}

\newcommand{\oli}{\overline}

\newcommand{\wti}{\widetilde}
\newcommand{\vphi}{\varphi}

\newcommand{\Si}{\Sigma}

\newcommand{\spt}{\mathrm{spt\,}}
\newcommand{\sptt}{\mathrm{spt}^2\,}

\newcommand{\dist}{\operatorname{dist}}

\DeclareMathOperator{\vol}{Vol}
\DeclareMathOperator{\area}{Area}
\DeclareMathOperator{\Index}{index}

\title[Morse index of min-max free boundary minimal hypersurfaces]{Min-max theory for free boundary minimal hypersurfaces II - general Morse index bounds and applications}

\date{\today}
\author{Qiang Guang}
\address{Mathematical Sciences Institute, The Australian National University, Canberra, ACT 2601, Australia}
\email{qiang.guang@anu.edu.au}

\author{Martin Man-chun Li}
\address{Department of Mathematics, The Chinese University of Hong Kong, Shatin, N.T., Hong Kong}
\email{martinli@math.cuhk.edu.hk}

\author{Zhichao Wang}
\address{Max-Planck Institute for Mathematics, Vivatsgasse 7, 
53111 Bonn, Germany}
\email{wangzhichaonk@gmail.com}

\author{Xin Zhou}
\address{Department of Mathematics, University of California Santa Barbara, Santa Barbara, CA 93106, USA}
\email{zhou@math.ucsb.edu}

\begin{document}
\begin{abstract}
For any smooth Riemannian metric on an $(n+1)$-dimensional compact manifold with boundary $(M,\partial M)$ where $3\leq (n+1)\leq 7$, we establish general upper bounds for the Morse index of free boundary minimal hypersurfaces produced by min-max theory in the Almgren-Pitts setting. We apply our Morse index estimates to prove that for almost every (in the $C^\infty$ Baire sense) Riemannan metric, the union of all compact, properly embedded free boundary minimal hypersurfaces is dense in $M$. If $\partial M$ is further assumed to have a strictly mean convex point, we show the existence of infinitely many compact, properly embedded free boundary minimal hypersurfaces whose boundaries are non-empty. Our results prove a conjecture of Yau for generic metrics in the free boundary setting.
\end{abstract}

\maketitle

\section{Introduction}
\label{S:Intro}

In his celebrated 1982 Problem Section, S.-T. Yau raised the following question:

\begin{question}[\cite{Yau82}]
\label{Q:Yau}
Does every closed three-dimensional Riemannian manifold $(M^3,g)$ contain infinitely many (immersed) minimal surfaces?
\end{question}

Back in 1960s, Almgren \cite{Alm62, Alm65} initiated an ambitious program to find minimal varieties, in arbitrary dimensions and codimensions, inside any compact Riemannian manifold (with or without boundary) using variational methods. He proved that weak solutions, in the sense of stationary varifolds, always exist. The interior regularity theory for codimension one hypersurfaces was developed about twenty years later by Pitts \cite{Pi} and Schoen-Simon \cite{SS}. As a consequence, they showed that in any closed manifold $(M^{n+1},g)$ there exists at least one embedded closed minimal hypersurface, which is smooth except possibly along a singular set of Hausdorff codimension at least 7. These fascinating results partly motivated \textbf{Question \ref{Q:Yau}} asked by Yau. In a very recent work \cite{LZ16}, the second and the last author developed a version of min-max theory for manifolds with boundary and proved up-to-the-boundary regularity for the free boundary minimal hypersurfaces produced by their theory, hence completing the program set out by Almgren in the hypersurface case. 

The foundational work of Almgren and Pitts left open an important (and very difficult) question of determining the Morse index of such minimal hypersurfaces produced by min-max methods. According to finite dimensional Morse theory, one expects the Morse index of the critical point produced using $k$-parameter families should be at most $k$. There had been no progress to this question for almost thirty years until the recent striking advances led by Marques and Neves. For instance, a precise control of the Morse index plays a significant role in their remarkable solution to the Willmore conjecture \cite{MN1}. Later in \cite{MN16}, they established the general Morse index upper bounds for minimal hypersurfaces produced by Almgren-Pitts theory for compact manifolds \emph{without boundary} (the one-parameter case was studied earlier in the works of Marques-Neves \cite{MN12} and the last author \cite{Zhou15, Zhou17}). In \cite{MN18}, they raised the question where similar index bounds hold in other min-max constructions as e.g. in \cite{CL16,DeRa18,DeTa,Mon16,ZZ17}. 

In this paper, we prove general Morse index upper bounds analogous to the ones in \cite{MN16} arising from the min-max theory developed in \cite{LZ16} for compact manifolds \emph{with boundary} in the Almgren-Pitts setting. Our main result can be stated roughly as follow:

\begin{theorem}[General index upper bound, simplified version]
If $\Sigma$ is the min-max free boundary minimal hypersurface in a compact Riemannian manifold with boundary produced by min-max over a $k$-dimensional homotopy class, then 
\[ \Index (\spt (\Sigma)) \leq k.\]
Here, the Morse index of the support of $\Sigma=m_1 \Sigma_1 + \cdots + m_k \Sigma_k$ is defined to be the sum of the indices of its components
\[ \Index (\spt(\Sigma)) := \Index (\Sigma_1) +\cdots + \Index (\Sigma_k).\] 
\end{theorem}

A more precise statement can be found in Theorem \ref{T:main} in Section \ref{S:main}. We will explain the technical aspects of our main result later towards the end of this section. Before that, we present a few applications of our general Morse index upper bounds.

\subsection*{A. Song's proof of Yau's conjecture for closed Riemannian manifolds}

In the last few years, we have witnessed substantial progress towards Yau's conjecture \textbf{Question \ref{Q:Yau}} leading to a complete solution by Song \cite{Song18} for \emph{closed} Riemannian manifolds. Marques and Neves \cite{MN17} made the first progress by settling Yau's conjecture for closed manifolds with positive Ricci curvature, or more generally, for closed manifolds satisfying the ``Embedded Frankel Property''. Song was able to localize the arguments in \cite{MN17} to produce infinitely many minimal hypersurfaces inside any domain $\Omega$ bounded by stable minimal hypersurfaces. The minimal hypersurfaces he constructed are limits of free boundary minimal hypersurfaces obtained from the min-max theory developed in \cite{LZ16}. As pointed out in \cite[\S 2.3]{Song18}, it was not known at that time whether the $p$-width $\omega_p(M,g)$ of a compact manifold with non-empty boundary is achieved by an integral varifold since the free-boundary analogue of the index bounds in \cite{MN16} was not yet available in the literature. Our general Morse index upper bounds together with the compactness theorem in \cite{GWZ18} provides the missing piece (see Proposition \ref{P:width}), hence simplifying some of the arguments in \cite{Song18}.

\subsection*{Density of free boundary minimal hypersurfaces for generic metrics}

Using the Weyl Law for the volume spectrum \cite{LMN16}, Irie, Marques and Neves \cite{IMN17} proved Yau's conjecture for generic metrics. In fact, they proved a stronger property that the union of all closed, smoothly embedded minimal hypersurfaces is dense in any closed manifold $M$ with a generic (in the $C^\infty$ Baire sense) Riemannian metric $g$. As an application of our main index estimates, we prove the same result for compact manifolds with boundary.

\begin{theorem}(Density of minimal hypersurfaces)
\label{T:density}
Let $(M^{n+1},\partial M)$ be a compact manifold with boundary and $3 \leq (n+1) \leq 7$. Then for a $C^\infty$-generic Riemannian metric $g$ on $M$, the union of all compact, properly embedded, free boundary minimal hypersurfaces is dense.
\end{theorem}

In particular, this provides an affirmative answer in the generic case to Yau's conjecture \textbf{Question \ref{Q:Yau}} for compact manifolds \emph{with boundary}. In fact, we will prove a much stronger property (Proposition \ref{conj:free:dense}) that there are infinitely many compact, properly embedded, free boundary minimal hypersurfaces intersecting any given relatively open set in $M$. Note that with the genericity assumption, we are even able to prove the density results within the smaller class of \emph{properly} embedded free boundary minimal hypersurfaces instead of the \emph{almost properly} embedded ones (see \cite[Definition 2.6]{LZ16} for the definitions). This is achieved by a perturbation argument in Proposition \ref{prop:perturb to be proper}.

\subsection*{Existence of minimal hypersurfaces with non-empty free boundary}

Our last application addresses a major problem in many of the constructions of free boundary minimal surfaces, as for example in \cite{Fra00, LZ16}, that the minimal surface produced may in fact be closed (i.e. \emph{without boundary}). In \cite{LZ16}, the second and the last author proved that there exist infinitely many \emph{properly} embedded free boundary minimal hypersurfaces, provided that the ambient manifold is compact with nonnegative Ricci curvature and strictly convex boundary (note that no \emph{closed} minimal hypersurfaces exist in such manifolds by \cite[Lemma 2.2]{FrLi}). Without any topological or curvature assumptions, it is in general very difficult to prevent the free boundary components from degenerating through the limit process (see e.g. \cite{ACS17} and \cite{GWZ18, Wang19}). Making use of our strong density result (Proposition \ref{conj:free:dense}), we are able to prove the same result by merely assuming \emph{strict mean convexity at one point} of the boundary $\partial M$ for a generic metric.

\begin{theorem}[Non-trivial free boundary]
\label{T:boundary-exist}
Let $(M^{n+1},\partial M)$ be a compact manifold with non-empty boundary and $3\leq (n+1)\leq 7$, equipped with a generic Riemannian metric $g$ as in Theorem \ref{T:density}. Suppose that $\partial M$ has a strictly mean convex point $x$, i.e. $H(x)>0$. Then there exist infinitely many distinct compact, smooth, properly embedded minimal hypersurfaces in $M$ with \emph{non-empty} free boundary. 
\end{theorem}

\begin{proof} We prove by contradiction. Assuming on the contrary that there are only finitely many distinct compact, smooth, properly embedded minimal hypersurfaces in $M$ with \emph{non-empty} free boundary, we can always take  a point $p\in\partial M$ and $r>0$ so that $H(p)>0$ and no such hypersurface intersects $B_r(p)$, where $B_r(p)$ is the geodesic ball in $M$ with center $p$ and radius $r$. 

Since $H(p)>0$,  the maximum principle of White (\cite[Theorem 1]{Wh10}) implies that there exists an $\epsilon>0$ (we may assume $\epsilon<r$) such that any minimal hypersurface $\Sigma$ having no boundary inside $B_r(p)$ must satisfy $\text{dist}(p,\Sigma)>\epsilon$. By Theorem \ref{T:density}, we know that there exists a properly embedded free boundary minimal hypersurface $\Gamma$ with $\Gamma \cap B_{\epsilon}(p)\neq \emptyset$. By the choice of $\epsilon$, it follow that $\Gamma$ has non-empty free boundary in $B_\epsilon(p)$. However, this contradicts our choice of $p$ and $r$. 
\end{proof}

\begin{remark}
From the proof of Theorem \ref{T:boundary-exist} we see that the union of the boundary of all compact, properly embedded, free boundary minimal hypersurfaces is also dense (as a subset of $\partial M$) in the strictly mean convex portion of $\partial M$. In particular, if $\partial M$ is strictly mean convex everywhere, then the union of their free boundaries is dense as a subset of $\partial M$.
\end{remark}




\subsection*{Main ideas of the proof}

We now explain the major ideas in proving the general index upper bound in Theorem \ref{T:main}. While our arguments follow the basic strategy of \cite{MN16}, the presence of a boundary, however, poses additional difficulties that have to be overcome. 

First of all, the notion of Morse index for a free boundary minimal hypersurface $\Sigma$ in $M$ is a subtle issue when $\Sigma$ is \emph{improper} (in other words, the touching set $(\Sigma \cap \partial M) \setminus \partial \Sigma$ is non-empty). Although it was shown in \cite{LZ16} that the hypersurface $\Sigma$ is smooth and minimal even across the touching set, a-priori it is not clear whether one should consider second variations moving the touching set. Motivated by the compactness result in \cite{GWZ18}, we argue that the appropriate notion of Morse index should only count the negative second variations supported away from the touching set. This is natural already at the level of first variation as the stationarity of a varifold in $M$ are taken with respect to the diffeomorphisms of $M$ generated by vector fields \emph{tangent to $\partial M$}. As $\Sigma$ is necessarily tangent to $\partial M$ along the touching set, any normal variations on $\Sigma$ should vanish on the touching set.

Second, the min-max theory for compact manifolds with boundary developed in \cite{LZ16} is formulated in terms of the space of (equivalence classes) of relative cycles consisting of \emph{integer-rectifiable} currents; while the Weyl Law for the volume spectrum estabilshed in \cite{LMN16} was formulated using the space of relative cycles consisting of \emph{integral} currents. The first formulation has the advantage that the regularity result of Gr\"{u}ter \cite{Gr87} can be readily deployed (the boundary regularity was irrelevant in the work of \cite{LMN16}) and the second formulation follows the original setup in Almgren's isomorphism theorem \cite{Alm62}. Since we need both results in this paper, we give a rigorous proof of the equivalence of two formulations in Section \ref{sec:equivalence}.

Third, the min-max theory in \cite{LZ16} is formulated in terms of sequences of maps defined on the vertices of finer and finer grids measured with respect to the mass topology. While this is essential for the regularity theory, for applications it is more useful to formulate a continuous min-max theory where maps are defined on the full parameter space and are continuous with respect to the $\mf F$-topology. In Section \ref{sec:min-max theory}, we explain how to use the original discretized setting to obtain a min-max theory for maps that are continuous in the $\mf F$-topology as above. A similar construction for the mass topology was done by the second and the last author in \cite{LZ16}. We prove the Min-Max Theorem (Theorem \ref{thm:minmax}) and Deformation Theorem (Theorem \ref{thm:deform}) as in \cite{MN16} in the continuous setting.

Finally, another crucial ingredient in proving the general Morse index bounds in \cite{MN16} is Sharp's generic finiteness result \cite{Sharp17} for minimal hypersurfaces with bounded index and area. Such a generic finiteness result was very difficult to establish for free boundary minimal hypersurfaces (see \cite{GWZ18, Wang19}) due to boundary degenerations and multiplicities. For our purpose of this paper, we only need the countability of free boundary minimal hypersurfaces with bounded index and area in any compact manifold with boundary equipped with a bumpy metric. We give a more direct inductive proof of this weaker result in \S \ref{subsec:counbtability}. Combining the Deformation Theorem with the generic countability result, we argue as in \cite{MN16} to prove the desired index bounds.

The organization of this paper is as follows. In Section \ref{S:main}, we set up some basic notations for the rest of the article and give a precise statement of our general Morse index bounds. In Section \ref{sec:equivalence}, we recall the two formulations of min-max theory for manifolds with boundary and prove their equivalence. In Section \ref{sec:min-max theory}, we describe the min-max theory in a continuous setting using the $\mf F$-topology and prove the Min-max theorem (Theorem \ref{thm:minmax}). In Section \ref{sec:defrom}, we prove the generic countability result (Proposition \ref{prop:countablity of W}) and the Deformation Theorem (Theorem \ref{thm:deform}), which is similar to \cite[Deformation Theorem A]{MN16} with slightly modifications. Using results in earlier sections, we prove our general Morse index upper bounds in Section \ref{sec:index estimate}. Finally in Section \ref{sec:density}, we give the proof of the density theorem (Theorem \ref{T:density}) as a corollary. In the appendix, we recall a construction for the logarithmic cut-off trick used in this paper.

\begin{acknowledgements}
We would like to thank Prof. Richard Schoen for his constant support and encouragement. The first author is partially supported  by an AMS-Simons Travel Grant. The second author is partially supported by a research grant from the Research Grants Council of the Hong Kong Special Administrative Region, China [Project No.: CUHK 14323516] and CUHK Direct Grant [Project Code:4053291].  Part of this work was done when the last author visited the Institute for Advanced Study under the support by the NSF Grant DMS-1638352, and he would like to thank IAS for their hospitality. The last author is also partially supported by NSF grant DMS-1811293 and an Alfred P. Sloan Research Fellowship.
\end{acknowledgements}


\section{Definitions and main results}
\label{S:main}

We now give a more precise statement of our general Morse index upper bounds in this section. Let $(M^{n+1},\partial M,g)$ be an $(n+1)$-dimensional compact Riemannian manifold with boundary and $3\leq (n+1)\leq 7$. Denote $G$ to be either the group  $\mb Z$ or $\mb Z_2$. Let $X$ be a simplicial complex of dimension $k$ and $\Phi: X\rightarrow \mc Z_n(M^{n+1},\partial M;\mf F; G)$ be a continuous map. Here, $\mc Z_n(M,\partial M; G)$ is the space of integer rectifiable $n$-currents $T$ in $M$ with coefficients in $G$ and $\partial T \subset \partial M$, modulo an equivalence relation (see \S \ref{subsec:integer rectifiable currents}). The notation $\mc Z_n(M,\partial M; \mf F; G)$ indicates that the space $\mc Z_n(M,\partial M; G)$ is endowed with the $\mf F$-topology, to be defined later in \S \ref{subsec:integer rectifiable currents}. Basically, continuity in $\mf F$-topology means that continuity in both the flat and the varifold topologies. 

Let $\Pi$ denote the set of all continuous maps $\Psi:X\to \mc{Z}_n(M,\partial M;\mf{F};G) $ such that $\Phi$ and $\Psi$ are homotopic to each other in the flat topology. The {\em width} of $\Pi$ is defined to be the min-max invariant:
\[\mf L(\Pi) := \inf_{\Psi\in\Pi} \sup_{x\in X} \{\mf M(\Psi(x))\},\]
where $\mf M(\tau)$ denotes the mass of the equivalence class $\tau=[T] \in \mc Z_n(M,\partial M; G)$, which is equal to the $n$-dimensional area of the {\em canonical representative} of $\tau$ (see \S \ref{subsec:integer rectifiable currents}). Given a sequence $\{\Phi_i\}_{i\in\mb N} \subset \Pi$ of continuous maps from $X$ into $\mc Z_n (M,\partial M;\mf F;G)$, we set
\[\mf L(\{\Phi_i\}_{i\in\mb N}):= \limsup_{i\rightarrow\infty}\sup_{x\in X}\mf M(\Phi_i(x)).\]
When $\mf L(\{\Phi_i\}_{i\in\mb N}) = \mf L(\Pi)$, we say $\{\Phi_i\}_{i\in\mb N}$ is a \emph{min-max sequence in the homotopy class $\Pi$}. 

Our main result below (Theorem \ref{T:main}) says that the width $\mf L(\Pi)$ can be realized by an integral varifold $V$ which is \emph{stationary in $M$ with free boundary} (see \cite[Definition 2.1]{LZ16}) and supported on a finite union $\cup_{j=1}^N \Sigma_j$ of smooth, almost properly embedded, free boundary minimal hypersurfaces $\Sigma_j$ whose sum of indices is at most $k$, the dimension of the parameter space $X$. Recall from \cite[Definition 2.6]{LZ16} that $(\Sigma,\partial \Sigma) \subset (M,\partial M)$ is an \emph{almost properly embedded} hypersurface if $\Sigma$ is an embedded hypersurface in $M$ whose $\partial \Sigma$ is contained in $\partial M$. Such a hypersurface is called a \emph{free boundary minimal hypersurface} (FBMH) if the mean curvature of $\Sigma$ vanishes and $\Sigma$ meets $\partial M$ orthogonally along $\partial \Sigma$. Given an almost properly embedded FBMH $\Sigma$ in $M$, there are several notions of Morse indices (see \cite{GWZ18}) on $\Sigma$. 
The quadratic form of $\Sigma$ associated to the
second variation formula is defined as
\[Q_\Sigma(v,v):=\int_\Sigma \left(|\nabla^\perp v|^2- \mathrm{Ric}_M(v,v)-|A^\Sigma|^2|v|^2 \right)\,d\mu_\Sigma  - \int_{\partial \Sigma} h^{\partial M}(v,v)\,d\mu_{\partial \Sigma},
\]
where $v$ is a section of the normal bundle of $\Sigma$, $\mathrm{Ric}_M$ is the Ricci curvature of $M$, $A^\Sigma$ and $h$ are the second fundamental forms of the hypersurface $\Sigma$ and $\partial M$, respectively. Note that we do not need to assume $\Sigma$ to be \emph{two-sided}. Denoting the \emph{touching set of $\Sigma$ in $M$} by 
\[ \mathrm{Touch}(\Sigma):=(\Sigma\cap\partial M)\setminus\partial \Sigma,\]
we define the \emph{Morse index of $\Sigma$}, denoted by $\Index(\Sigma)$, as the maximal dimension of a linear subspace of sections of normal bundle $N\Sigma$ \emph{compactly supported in $\Sigma\setminus\mathrm{Touch}(\Sigma)$} such that the quadratic form $Q_\Sigma(v,v)$ is negative definite on this subspace. (Note that this is the same notion as the \emph{Morse index of $\Sigma$ on the proper subset} defined in \cite{GWZ18}.)

With the definitions above (and the notions regarding varifolds in \S \ref{sec:equivalence}), we can now state our main theorem on the general Morse index upper bounds for the min-max minimal hypersurfaces with free boundary. 
\begin{theorem}[General index upper bounds]
\label{T:main}
Let $(M^{n+1},g)$ be a smooth compact $(n+1)$-dimensional Riemannian manifold with boundary and $3 \leq (n+1) \leq 7$. Let $X$ be a simplicial complex of dimensional $k$ and $\Phi:X\to \mc{Z}_n(M,\partial M;\mf{F};G)$ be a continuous map. Denote $\Pi$ as the associated homotopy class of $\Phi$ with respect to the flat topology. Then there exists a integral varifold $V\in \mc{V}_n(M)$ such that 
\begin{enumerate}[label=(\roman*)]
\item $\|V\|(M)=\mf{L}(\Pi)$;
\item $V$ is stationary in $M$ with free boundary;
\item there exists $N\in \mb{N}$ and $m_i\in \mb{N}$, $1\leq i\leq N$, such that $V=\sum_{i=1}^N m_i|\Sigma_i|$, where each $\Sigma_i$ is a smooth, compact, connected, almost properly embedded free boundary minimal hypersurface in $M$. Moreover, 
\[\sum_{i=1}^N \Index(\Sigma_i)\leq k.\]
\end{enumerate}
\end{theorem}

We will separate the proof of Theorem \ref{T:main} into two parts. The first part (Theorem \ref{thm:minmax}) handles the regularity theory for the min-max theory in a continuous setting. The second part (Theorem \ref{thm:index estimates}) gives the required index bounds.

\section{Equivalence of two formulations}
\label{sec:equivalence}

In this section, we describe the two different formulations of min-max theory for manifolds with boundary introduced in \cite{LZ16} and \cite{LMN16}. The goal is to show that the two formulations are equivalent. The arguments are simple but a bit tedious. Readers can skip this section and refer back to the definitions later if necessary.

Let $(M^{n+1},g)$ be a smooth compact connected Riemannian manifold with nonempty boundary $\partial M$. Without loss of generality, we can regard $M$ as a compact domain of a closed Riemannian manifold $\wti{M}$ of the same dimension, which is isometrically embedded into $\mb{R}^L$ for some $L$ large enough. We recall some basic notations in geometric measure theory essentially following \cite{LZ16}.

We use $\mc{V}_k(M)$ to denote the closure of the space of $k$-dimensional rectifiable varifolds in $\mb{R}^L$ with support contained in $M$. Let $G$ be either the group $\mb{Z}$ or $\mb{Z}_2$. Let $\mc{R}_k(M;G)$ (resp. $\mc{R}_k(\partial M;G)$) be the space of $k$-dimensional rectifiable currents in $\mb{R}^L$ with coefficients in $G$ which are supported in $M$ (resp. in $\partial M$). Denote by $\spt T$ the support of $T\in\mc R_k(M;G)$. Given any $T\in \mc{R}_k(M;G)$, denote by $|T|$ and $\|T\|$ the integer rectifiable varifold and the Radon measure in $M$ associated with $T$, respectively. The mass norm and the flat metric on $\mc{R}_k(M;G)$ are denoted by $\mf{M}$ and $\mc{F}$ respectively; see \cite{Fed69}. As we can regard any $T\in\mc R_k(M;\mb Z_2)$ as an element in $\mc R_k(M;\mb Z)$, we use $\sptt T$ and $\mathrm{spt}^0\, T$ to denote the support of $T$ when regarded as an elements in $\mc R_k(M;\mb Z_2)$ and $\mc R_k(M;\mb Z)$ respectively. Similarly, we use $\mf M^2(T)$ and $\mf M^0(T)$ to denote their respective mass norms. 

\subsection{Formulation using integer rectifiable currents}
\label{subsec:integer rectifiable currents}

We now recall the formulation in \cite{LZ16} using equivalence classes of integer rectifiable currents. Let 
\begin{equation}
\label{eq:def from integer rect}
 Z_k(M,\partial M;G):=\{T\in \mc{R}_k(M;G) : \spt(\partial T)\subset \partial M \}.
 \end{equation}
We say that two elements $T,S\in Z_k(M,\partial M;G)$ are equivalent if $T-S\in \mc{R}_k(\partial M;G)$. We use $\mc{Z}_k(M,\partial M; G)$ to denote the space of all such equivalence classes. For any $\tau\in \mc{Z}_k(M,\partial M; G)$, we can find a unique $T\in \tau$ such that $T\llcorner \partial M=0$. We call such $T$ the \emph{canonical representative} of $\tau$ as in \cite{LZ16}. For any $\tau\in \mc{Z}_k(M,\partial M;G)$, its mass and flat norms are defined by 
\[\mf{M}(\tau):=\inf\{\mf{M}(T) : T\in \tau\}\quad \text{ and }\quad \mc{F}(\tau):=\inf\{\mc{F}(T) : T\in \tau\}.\]
The support of $\tau\in \mc{Z}_k(M,\partial M;G)$ is defined by 
\[\spt (\tau):=\bigcap_{T\in \tau}\spt (T).\]
By \cite[Lemma 3.3]{LZ16}, we know that for any $\tau \in \mc{Z}_k(M,\partial M;G)$, we have $\mf{M}(T)=\mf{M}(\tau)$ and $\spt (\tau)=\spt (T)$, where $T$ is the canonical representative of $\tau$.

Recall that the varifold distance function $\mf{F}$ on $\mc{V}_k(M)$ is defined in \cite[2.1 (19)]{Pi}, which induces the varifold weak topology on the set $\mc{V}_k(M)\cap \{V : \|V\|(M)\leq c\}$ for any $c$. We also need the $\mf{F}$-metric on $\mc{Z}_k(M,\partial M;G)$ defined as follows: for any $\tau, \sigma\in \mc{Z}_k(M,\partial M;G)$ with canonical representatives $T\in \tau$ and $S\in \sigma$, the $\mf{F}$-metric of $\tau$ and $\sigma$ is 
\[\mf{F}(\tau,\sigma):=\mc{F}(\tau-\sigma)+\mf{F}(|T|,|S|),\]
where $\mf{F}$ on the right hand side denotes the varifold distance on $\mc{V}_k(M)$. 

For any $\tau\in \mc{Z}_k(M,\partial M; G)$, we define $|\tau|$ to be $|T|$, where $T$ is the unique canonical representative of $\tau$ and $|T|$ is the rectifiable varifold corresponding to $T$.

We assume that $\mc{Z}_k(M,\partial M;G)$ have the flat topology induced by the flat metric. With the topology of mass norm or the $\mf{F}$-metric, the space will be denoted by  $\mc{Z}_k(M,\partial M;\mf{M};G)$ or  $\mc{Z}_k(M,\partial M;\mf{F};G)$.

\subsection{Formulation using integral currents}\label{subsec:using integral}

We now recall the formulation in \cite{LMN16} using equivalence classes of integral cycles. For $k\geq 1$, let $\mf I_k (M;\mb Z_2)$ denote those elements of $\mc R_k(M;\mb Z_2)$ whose boundary lies in $\mc R_{k-1}(M;\mb Z_2)$. The space $\mf I_k(\partial M;\mb Z_2)$ is defined similarly with $M$ replaced by $\partial M$. We also consider the space
\begin{equation}\label{eq:def from integral currents}
 Z_{k,rel}(M,\partial M;\mb Z_2):=\{T\in \mf I_k(M;\mb Z_2):\sptt \partial T\subset \partial M \}, 
 \end{equation}
endowed with the flat topology. We say that two elements $T, S \in Z_{k,rel}(M,\partial M;\mb Z_2)$ are equivalent if $T-S\in\mf I_k(\partial M;\mb Z_2)$ and the space of such equivalence classes is denoted by $\mc Z_{k,rel}(M,\partial M;\mb Z_2)$. The mass norm and the flat metric on this space are defined respectively as follows:
\[\mf{M}(\tau):=\inf\{\mf{M}(T) : T\in \tau\}\quad \text{ and }\quad \mc{F}(\tau):=\inf\{\mc{F}(T) : T\in \tau\},\]
where $\tau \in Z_{k,rel}(M,\partial M;\mb Z_2)$.

\subsection{Equivalence of two formulations}

We now prove that the spaces $\mc Z_n(M,\partial M;\mb Z_2)$ and $\mc Z_{n,rel}(M,\partial M;\mb Z_2)$ defined in (\ref{eq:def from integer rect}) and (\ref{eq:def from integral currents}) are isomorphic to each other. We begin by stating a preliminary lemma, which was proven in \cite{LZ16} for $\mb Z$ coefficients. The same proof actually works for $\mb Z_2$ coefficients as well (c.f. \cite[Theorem 2.3]{LMN16}).

\begin{lemma}[cf. {\cite[Lemma 3.8]{LZ16}}]
\label{lemma:appro by integral}
Given $T\in Z_k(M,\partial M;\mb Z_2)$, there exists a sequence $T_i\in Z_{k,rel}(M,\partial M;\mb Z_2)$ such that $T_i-T\in\mc R_{k}(\partial M;\mb Z_2)$ and  $\lim_{i\rightarrow\infty}\mf M(T_i-T)=0$ (and thus $\lim_{i\to \infty} \mc F(T_i-T)=0$).
\end{lemma}

Now we prove the main result of this section.

\begin{proposition}
The spaces $\mc Z_n(M,\partial M;\mb Z_2)$ and $\mc Z_{n,rel}(M,\partial M;\mb Z_2)$, endowed with the flat topology, are homeomorphic. In fact, they are isometric with respect to $\mc F$ and $\mf M$.
\end{proposition} 

\begin{proof}
{\bf Step I: Construction of a map $\iota$ from $\mc Z_{n,rel}(M,\partial M;\mb Z_2)$ to $\mc Z_n(M,\partial M;\mb Z_2)$.}

Let $\tau \in \mc Z_{n,rel}(M,\partial M; \mb Z_2)$ and $T\in\tau$. Then $T$ is an element in $Z_n(M,\partial M;\mb Z_2)$. Moreover, if $T'$ and $T$ are equivalent in $Z_{n,rel}(M,\partial M;\mb Z_2)$, they are also equivalent in $Z_n(M,\partial M;\mb Z_2)$. Thus we get a well-defined map $\iota:\mc Z_{n,rel}(M,\partial M;\mb Z_2) \to \mc Z_n(M,\partial M;\mb Z_2)$.

\vspace{0.5em}
{\bf Step II: Construction of a map $\eta$ from $\mc Z_n(M,\partial M;\mb Z_2)$ to $\mc Z_{n,rel}(M,\partial M;\mb Z_2)$.}
 
Let $\kappa\in \mc Z_n(M,\partial M;\mb Z_2)$, Lemma \ref{lemma:appro by integral} gives an element $S\in \kappa$ such that $\partial S\in\mc R_{n-1}(\partial M;\mb Z_2)$. Hence $S\in Z_{n,rel}(M,\partial M;\mb Z_2)$ by definition in \S \ref{subsec:using integral}. Denote by $\tau_S$ the equivalence class of $S$ in $\mc Z_{n,rel}(M,\partial M;\mb Z_2)$. Define the map \[\eta: \mc Z_{n}(M,\partial M;\mb Z_2) \to \mc Z_{n,rel}(M,\partial M;\mb Z_2)\] by $\eta(\kappa):=\tau_S$. We now prove that $\eta$ is well-defined. Suppose we have another $S' \in \kappa$ satisfying $\partial S' \in\mc R_{n-1}(\partial M;\mb Z_2)$, then we have $S-S' \in \mf I_n(M;\mb Z_2)$. Recall that $S,S' \in \kappa$ implies that $\spt (S-S')\subset \partial M$. Hence $S- S' \in \mf I_n(\partial M;\mb Z_2)$, which means that $S$ and $S'$ are equivalent in $Z_{n,rel}(M,\partial M;\mb Z_2)$. Thus $\eta$ is well-defined.

\vspace{0.5em}
{\bf Step III: Check that $\iota$ and $\eta$ are inverses to each other.}
Indeed, take $\tau\in \mc Z_{n,rel}(M,\partial M;\mb Z_2)$ and $T\in \tau$. Note that $T \in \mf I_n(M; \mb Z_2)$. By definition, we have $\tau=\tau_T=\eta(\iota(\tau))$. Hence, $\eta\circ\iota=\mathrm{id}$. To prove that $\iota\circ\eta=\mathrm{id}$, we take $\kappa\in \mc Z_n(M,\partial M;\mb Z_2)$ and $S\in\kappa$ such that $\partial S\in\mc R_{n-1}(\partial M;\mb Z_2)$. Note that $S\in \tau_S$. Therefore, $\iota(\eta(\kappa))=\iota(\tau_S)=\kappa$.

\vspace{0.5em}
{\bf Step IV: Show that $\iota$ and $\eta$ are isometries with respect to $\mc F$ and $\mf M$.}

It is clear from the definitions that both $\iota$ and $\eta$ are linear maps. Let $\nu=\mc F$ or $\mf M$. Taking $\tau\in\mc Z_{n,rel}(M,\partial M;\mb Z_2)$ and $T_j\in\tau$ so that $\nu(T_j)\rightarrow\nu(\tau)$, we observe that $\nu(T_j)\geq \nu(\iota(\tau))$ since $T_j\in\iota(\tau)$. It follows that $\nu(\iota(\tau))\leq \nu(\tau)$. On the other hand, taking $\kappa \in\mc Z_n(M,\partial M;\mb Z_2)$ and $S_j\in\kappa$ so that $\nu(S_j)\rightarrow\nu(\kappa)$, by Lemma \ref{lemma:appro by integral} there exists a sequence $S_j'\in Z_{n,rel}(M,\partial M;\mb Z_2)$ so that $S_j'-S_j\in\mc R_n(\partial M;\mb Z_2)$ and $\mf M(S_j'-S_j)<1/j$. From the definition of $\eta$, we have $S'_j\in\eta(\kappa)$. Thus, 
\[\nu(\eta(\kappa))\leq \lim_{j\rightarrow\infty}\nu(S_j')=\lim_{j\rightarrow\infty}\nu(S_j)=\nu(\kappa).\]
Together with $\eta\circ\iota=\mathrm{id}$ and $\iota\circ\eta=\mathrm{id}$, we conclude that $\nu(\kappa)=\nu(\eta(\kappa))$ and $\nu(\tau)=\nu(\iota(\tau))$.
\end{proof}

\section{Min-max theory in continuous setting}
\label{sec:min-max theory}

In this section, we describe the min-max theory for manifolds with boundary in continuous setting. Recall that in \cite{LZ16} the Almgren-Pitts min-max theory for compact manifolds with boundary deals with discrete families of elements in $\mc{Z}_k(M,\partial M;G)$. The essential tools  connecting the discrete and continuous settings are the discretization theorem (see \cite[Theorem 4.12]{LZ16}) and the interpolation theorem (see \cite[Theorem 4.14]{LZ16}).

Let $I^m=[0,1]^m$ denote the $m$-dimensional cube. Suppose that $X$ is a subcomplex of dimension $k$ of $I^m$. We adopt the notations for cell complex as in Definition 4.1 of \cite{LZ16}. In particular, the cube complex $X(j)$ denotes the set of all cells of $I(m,j)$ whose support is contained in some cell of $X$, and $X(j)_p$ denotes the set of all $p$-cells in $X(j)$.

Suppose that $\Phi:X\to \mc{Z}_n(M,\partial M;\mf{F};G)$ is a continuous map with respect to the $\mf{F}$-metric. We use $\Pi$ to denote the set of all continuous maps $\Psi:X\to \mc{Z}_n(M,\partial M;\mf{F};G) $ such that $\Phi$ and $\Psi$ are homotopic to each other in the flat topology.
\begin{definition}(Width and min-max sequences) The \emph{width} of $\Pi$ is defined by 
\[\mf{L}(\Pi)=\inf_{\Phi\in \Pi}\sup_{x\in X} \mf{M}(\Phi(x)).\]
A sequence $\{\Phi_i\}_{i\in \mb{N}}\subset \Pi$ is called a \emph{min-max sequence} if $\mf{L}(\Phi_i)=\sup_{x\in X}\mf{M}(\Phi_i(x))$ satisfies
\[\mf{L}(\{\Phi_i\}_{i\in \mb{N}})=\limsup_{i\to \infty}\mf{L}(\Phi_i)=\mf{L}(\Pi).\]
\end{definition}

\begin{definition}(Critical set)
The \emph{image set} of $\{\Phi_i\}_{i\in \mb{N}}$  is defined by 
\[ \begin{split}  \mf{\Lambda}(\{\Phi_i\}_{i\in \mb{N}})=\{ & V\in \mc{V}_n(M): \exists\, \text{  sequences }\, \{i_j\}\to \infty, x_{i_j}\in X\,\\ & \text{ such that }\, \lim_{j\to \infty}\mf{F}(|\Phi_{i_j}(x_{i_j})|, V)=0 \}. \end{split} \]
Let $\{\Phi_i\}_{i\in \mb{N}}$ be a min-max sequence in $\Pi$ such that $L=\mf{L}(\{\Phi_i\}_{i\in \mb{N}})$. The \emph{critical set} of $\{\Phi_i\}_{i\in \mb{N}}$ is defined by 
\[\mf{C}(\{\Phi_i\}_{i\in \mb{N}})=\{V\in \mf{\Lambda}(\{\Phi_i\}_{i\in \mb{N}}): \|V\|(M)=L  \}. \]
\end{definition}

Note that for any min-max sequence $\{\Phi_i\}_{i\in \mb{N}}\subset \Pi$, by the tightening construction (see \cite[Proposition 4.17]{LZ16}), we can find another min-max sequence $\{\Phi_i'\}_{i\in \mb{N}}\subset \Pi$ such that $ \mf{C}(\{\Phi_i'\}_{i\in \mb{N}})\subset \mf{C}(\{\Phi_i\}_{i\in \mb{N}})$ and each $V\in \mf{C}(\{\Phi_i'\}_{i\in \mb{N}})$ is {\em stationary in $M$ with free boundary} (see \cite[Defintion 2.1]{LZ16}).
\begin{lemma} \label{lemma:no mass}
If $\Phi:X\to \mc{Z}_n(M,\partial M;\mf{F};G)$  is a continuous map with respect to the $\mf{F}$-metric, then $\Phi$ has no concentration of mass.
\end{lemma}
\begin{proof}
This lemma follows from the definition directly (c.f. \cite[p.472]{MN16}).
\end{proof}

The next theorem tells us that we can construct a continuous map in the $\mf M$-norm out of a discrete map with small fineness.

\begin{theorem}(Interpolation Theorem; \cite[Theorem 4.14]{LZ16}, \cite[Theorem 2.11]{LMN16})
\label{thm:interpolation}
Let $M^{n+1}$ be a compact Riemannian manifold with boundary and $m\in \mb{N}$. Then there exists $C_0>0$ and $\delta_0>0$ depending only on $M$ and $m$ such that if $X$ is a cubical subcomplex of $I(m,l)$ and 
\[\phi: X_0\to \mc{Z}_n(M,\partial M; G)\]
has $\mf{f}_{\mf{M}}(\phi)<\delta_0$ \cite[Definition 4.2]{LZ16}, then there exists a map 
\[\Phi: X\to \mc{Z}_n(M,\partial M; G)\] which is continuous in the $\mf{M}$-topology and satisfying 
\begin{enumerate}[label=(\roman*)]
\item $\Phi(x)=\phi(x)$ for all $x\in X_0$;
\item for any $p$-cell $\alpha$ in $X_p$, $\Phi|_{\alpha}$ depends only on the restriction of $\phi$ on the vertices of $\alpha$ and 
\[\max\{\mf{M}(\Phi(x)-\Phi(y)) : x,y\in \alpha\}\leq C_0\mf{f}_{\mf{M}}(\phi).\]
\end{enumerate}
\end{theorem}
We remark that in the above theorem, the map $\Phi$ is called the \emph{Almgren extension} of $\phi$.
\vspace{2mm}

Next, we will formulate the min-max theory for manifolds with boundary in continuous setting (cf. \cite[Theorem 3.8]{MN16}).

\begin{theorem}(Min-max Theorem)
\label{thm:minmax}
Let $3 \leq n+1 \leq 7$. 
Suppose that $\mf{L}(\Pi)>0$ and $\{\Phi_i\}_{i\in \mb{N}}\subset \Pi$ is a min-max sequence. Then there exists a varifold $V\in \mf{C}(\{\Phi_i\}_{i\in \mb{N}})$ such that 
\begin{itemize}
\item[(i)] $\|V\|(M)=\mf{L}(\Pi)$;
\item[(ii)] $V$ is stationary in $M$ with free boundary;
\item [(iii)] $V$ is supported on a smooth, compact, almost properly embedded free boundary minimal hypersurface. 
\end{itemize}
\end{theorem} 
\begin{proof}

The proof is parallel to the one in \cite[Theorem 3.8]{MN16}. Here, we adapt the arguments in the proof of \cite[Theorem 3.8]{MN16} and make necessary modifications. 
By the tightening construction, we may assume that every element of $\mf{C}(\{\Phi_i\}_i)$ is stationary with free boundary.

\vspace{2mm}
\textbf{Step one: The discretization process.} 
For any $\Phi_i:X\to \mc{Z}_n(M,\partial M;\mf{F};G)$, by Lemma \ref{lemma:no mass}, we know that $\Phi_i$ has no concentration of mass. By the Discretization Theorem, \cite[Theorem 4.12]{LZ16} (see also \cite[Theorem 2.10]{LMN16}),  there is a sequence of maps
\[\phi_i^j: X(k_j^i)_0 \to \mc{Z}_n(M,\partial M; G)\]
with $k_j^i<k_{j+1}^i$ and a sequence of positive constants $\delta_j^i\to \infty$ as $j\to \infty$ such that 

\begin{enumerate}
\item $S_i=\{\phi_i^j\}_{j\in \mb{N}}$ is an $(X,\mf{M})$-homotopy sequence of mappings into $\mc{Z}_n(M,\partial M; G)$ with fineness $\mf{f}_{\mf{M}}(\phi_i^j)<\delta_j^i$;
\item \[\sup\{\mc{F}(\phi_i^j(x)-\Phi_i(x))\,:\, x\in X(k_j^i)_0\}\leq \delta_j^i;\]
\item 
\[   \sup\{\mf{M}(\phi_i^j(x)) : x\in X(k_j^i)_0\}\leq \sup\{\mf{M}(\Phi_i(x)) : x\in X\}+\delta_j^i;   \] 
\item there exists a sequence $l_{j}^i\to \infty$ as $j\to \infty$ such that for any $y\in X(k_j^i)_0$, 
\[\mf{M}(\phi_i^j(y))\leq \sup\{\mf{M}(\Phi_i(x)):\alpha \in X(l_j^i),x,y\in \alpha\}+\delta_j^i.\]
\end{enumerate}

Since $\Phi_i$ is continuous with respect to the $\mf{F}$-metric, we have that $x\in X\mapsto \mf{M}(\Phi_i(x))$ is continuous. Then property (4) implies that there exists $\eta_j^i\to 0$ as $j\to \infty$ such that for any $y\in X(k_j^i)_0$,
\begin{equation}\label{equ:minmax:mass}
\mf{M}(\phi_i^j(y))\leq \mf{M}(\Phi_i(y))+\eta_j^i .
\end{equation} 
By \cite[Lemma 3.13]{LZ16}, for any $\tau, \tau_k\in \mc{Z}_n(M,\partial M; G)$, $k\in \mb{N}$, then $\mf{F}(\tau,\tau_k)\to 0$ if and only if $\mc{F}(\tau,\tau_k)\to 0$ and $\mf{M}(\tau_k)\to \mf{M}(\tau)$ as $k\to \infty$. This together with property (2), (\ref{equ:minmax:mass}) and a standard compactness argument gives that 
\begin{equation}\label{equ:minmax:F}
\sup\{\mf{F}(\phi_i^j(x),\Phi_i(x)) : x\in X(k_j^i)_0\}\to 0\quad \text{ as }\, j\to \infty.
\end{equation}

\vspace{2mm}

\textbf{Step two: The diagonal argument.} Now combining properties (1)-(4) and (\ref{equ:minmax:F}), for each $i$, we can choose $j(i)\to \infty$ as $i\to \infty$ such that the diagonal sequence 
\[\vphi_i=\phi_i^{j(i)}: X(k_{j(i)}^i)_0\to \mc{Z}_n(M,\partial M; G)\]
satisfies 

\begin{itemize}
\item $\sup\{\mf{F}(\vphi_i(x),\Phi_i(x)): x\in X(k_{j(i)}^i)_0\}\leq a_i \quad \text{with }\, \lim_{i\to \infty}a_i=0$;
\item $\sup\{\mf{F}(\Phi_i(x),\Phi_i(y)) : x,y\in \alpha, \alpha\in X(k_{j(i)}^i)\}\leq a_i;$
\item the fineness $\mf{f}_{\mf{M}}(\vphi_i)$ tends to zero as $i\to \infty$.
\end{itemize}
Moreover, using the interpolation result (Theorem \ref{thm:interpolation}), we can approximate $\vphi_i$ by a continuous map $\Psi_i^{j(i)}: X\to \mc{Z}_n(M,\partial M; \mf{M};G)$, which is called the Almgren extension of $\vphi_i$. By property (2) and \cite[Proposition 2.12]{LMN16}, we can also require that $\Psi_i^{j(i)}$ is homotopic to $\Phi_i$ in the flat topology. 

Now we consider the sequence $S=\{\vphi_i\}_{i\in \mb{N}}$ and we have 
\[\mf{L}(S)=\limsup_{i\to \infty} \max\{\mf{M}(\vphi_i(x)): x\in X(k_{j(i)}^i)_0\}.\]
By (\ref{equ:minmax:mass}), we have $\mf{L}(S)\leq \mf{L}(\{\Phi_i\}_i)=\mf{L}(\Pi)$. Since the Almgren extension $\Psi_i^{j(i)}$ is homotopic to $\Phi_i$, we have $\Psi_i^{j(i)}\in \Pi$. This implies that $\mf{L}(\{\Phi_i\}_i)=\mf{L}(\Pi)\leq \mf{L}(S)$. Hence,  \[\mf{L}(S)=\mf{L}(\{\Phi_i\}_i)=\mf{L}(\Pi).\]
We can also prove that $\mf{C}(S)=\mf{C}(\{\Phi_i\}_i)$. By the choice of $\mf{C}(\{\Phi_i\}_i)$, we know that every element of $\mf{C}(S)$ is stationary with free boundary. 

Next, we claim that there exists an element $V\in \mf{C}(S)$ such that $V$ is almost minimizing in small annuli with free boundary (see \cite[Definition 4.19]{LZ16}). Otherwise, we can deform $S$ homotopically to another sequence $\wti{S}$ as in the proof of \cite[Theorem 4.21]{LZ16} such that $\mf{L}(\wti{S})<\mf{L}(S)$, which is a contradiction. 

Finally, the conclusion follows directly from the regularity results in \cite[Theorem 5.2]{LZ16} (when $G=\mb Z$) and Theorem \ref{thm:regularity of Z2 almost minimizing} (when $G=\mb Z_2$). This completes the proof.  
\end{proof}

Based on the work of Pitts \cite{Pi}, Schoen-Simon \cite{SS} and Gr\"uter \cite{Gr87}, Li-Zhou \cite{LZ16} proved the regularity for stationary varifolds  which is $\mb Z$-almost minimizing in small annuli with free boundary. As we explain below, the arguments also extend to $\mb Z_2$-coefficients.

\begin{theorem}[Regularity of $\mb Z_2$-almost minimizing varifolds]
\label{thm:regularity of Z2 almost minimizing}
Let $2\leq n\leq 6$. Suppose $V\in\mc V_n(M)$ is a varifold which is
\begin{itemize}
\item stationary in $M$ with free boundary and
\item $\mathbb Z_2$-almost minimizing in small annuli with free boundary,
\end{itemize}
then there exists $N\in\mb N$ and $n_i\in\mb N, i = 1,..., N $, such that
\[V=\sum_{i=1}^Nn_i |\Sigma_i|,\]
where each $(\Sigma_i,\partial\Sigma_i)\subset (M, \partial M )$ is a smooth, compact, connected, almost properly embedded free boundary minimal hypersurface.
\end{theorem}
\begin{proof}
The interior regularity of $V$ follows from \cite[Theorem 2.11]{MN17}. The regularity of $V$ on the boundary follows from a similar procedure as the proof in \cite{LZ16}. The only difference is to show the regularity of \emph{replacements} (see \cite[Proposition 5.3]{LZ16} for definition) for a $\mb Z_2$-almost minimizing varifold with free boundary. 
The last statement follows from \cite[Lemma 5.5]{LZ16} (where regularity of replacements for $\mb Z$-almost minimizing varifold with free boundary was proved) by replacing \cite[Regularity Theorem 2.4]{Mor86} with Theorem \ref{thm:regualrity of Z2 minimizer}.
\end{proof}

Finally, we establish the regularity for locally area-minimizers with respect to $\mb Z_2$ coefficients. A similar result was obtained earlier by Gr\"{u}ter \cite{Gr87} for $\mb Z$ coefficients.

\begin{theorem}[Regularity of $\mb Z_2$-minimizers]
\label{thm:regualrity of Z2 minimizer}
Let $S\subset\mb R^{n+1}$ be an $n$-dimensional submanifold of class $C^2$ and let $U$ be an open set such that $\partial S\cap U=\emptyset$. Suppose $T\in\mc R_n(\mb R^{n+1};\mb Z_2)$ with $\sptt T\subset \overline U$ and $\sptt\partial T\cap U\subset S$ such that 
\begin{equation}\label{eq:thm:area minimizer }
\mf M^2_W(T)\leq \mf M^2_W(T+X)
\end{equation}
for all open $W\subset\subset U$ and $X\in\mc R_n(U;\mb Z_2)$ with $\sptt X\subset W$ and $\sptt \partial X \cap U\subset S$. Then we have
\begin{itemize}
\item $\mathrm{sing}(T)=\emptyset$ if $2\leq n\leq 6$;
\item $\mathrm{sing}(T)$ is discrete if $n=7$;
\item $\mathrm{dim}(\mathrm{sing}(T))\leq n-7$ if $n>7$.
\end{itemize}
In case $x\in S\cap \mathrm{reg}(T)$ we know that $S$ and $T$ intersect orthogonally in a neighborhood of $x$.
\end{theorem}

\begin{proof}
It suffices to consider the regularity for $p\in S\cap \sptt T$. Denote by $B_r=B_r(p)$. Take $r$ small enough so that $B_r(p)\cap S$ is a $n$-ball and separates $B_r$ into $B^+_{r}$ and $B^-_r$. Without loss of generality, we can assume that $\sptt T\cap B_r^-=\emptyset$ (see \cite[\S 3]{Gr87} for more details).

For $x\in B_r$, denote by $\sigma$ the reflection across $S$ (see \cite[Remark 3.1]{Gr87}). 
We have $\sigma^2=\mathrm{id}$ and set
\[\wti T=T-\sigma_\#T.\]
Then $\sptt \wti T\subset V_r$ and $\sptt \partial \wti T\subset \partial V_r$, where $V_r=B_r^+\cup (S\cap B_r)\cup \sigma(B_r^+)$.
Furthermore, $B_{r/3}\subset V_r$ and $\mf M^2(\wti T)<\infty$. Then by Slicing Lemma \cite[\S 28]{Si}, we can take $\alpha\in(1/4,1/3)$ so that 
\[T':=\wti T\llcorner B_{\alpha r}\in \mf I_n(B_{\alpha r};\mb Z_2).\]

Recall that $T_1\in \mc R_n(B_{\alpha r};\mb Z)$ is {\em a representative modulo 2} (see \cite[Page 227]{Mor84}) of $T'$ if $T_1$ is of multiplicity one such that $T_1 = T'$ (mod 2) and $\mf M(T_1)=\mf M^2(T')$. 

Now since $\partial B_{\alpha}$ is simply connected and $\mc H^n(\sptt\partial T')=0$, applying \cite[Lemma 4.2]{Mor84} (letting $X_1=T'$ and $X_2$ be half of $\partial B_{\alpha}$ separated by $\sptt\partial T'$ therein), $T'$ has a representative modulo two, denoted as $T_1$, so that $\mathrm{spt}^0\,  \partial T_1=\mathrm{spt}^2\,\partial T'$. Set 
\[R:=T_1\llcorner B^+_{\alpha r}.\]

We now prove that $R$ is a $\mb Z$-area minimizer, and hence the regularity follows from \cite[Theorem 4.7]{Gr87}. To prove this, let $W\subset\subset B_{\alpha r}$ be an open set and $X\in \mc R_n(B_{\alpha r};\mb Z)$ with $\spt X\subset W$ and $\spt\partial X\cap U\subset S$. Then by the definition of $\mf M^2$, we have
\[\mf M^0_W(R+X)\geq \mf M^2_W(R+X)=\mf M^2_W(X+R\llcorner B^+_{\alpha r} ).\]
As we can see $R=T'\llcorner B^+_{\alpha r}=\wti T\llcorner B^+_{\alpha r}=T\llcorner B_{\alpha r}$ (mod 2), the inequality becomes
\begin{equation}\label{eq:R to T}
\mf M^0_W(R+X)\geq \mf M^2_W(X+T\llcorner B_{\alpha r}).
\end{equation}
Note that (\ref{eq:thm:area minimizer }) gives that 
\[\mf M^2_W(X+T\llcorner B_{\alpha r})\geq \mf M^2_W(T\llcorner B_{\alpha r})=\mf M^0_W(R).\]
Together with (\ref{eq:R to T}), we conclude that $R$ is area minimizing in the sense of \cite{Gr87} and we are done.
\end{proof}

\section{Deformation Theorem}
\label{sec:defrom}

Our goal in this section is to prove the Deformation Theorem (Theorem \ref{thm:deform}) which is a crucial ingredient in the proof of the Morse index estimates in the next section. We also establish a few preliminary results including the generic countability of FBMHs with bounded index and area and the notion of instability in the context of varifolds.

\subsection{Generic countability}
\label{subsec:counbtability}

As remarked in the introduction, the generic finiteness result for FBMH was proved in \cite{GWZ18} and \cite{Wang19}. The construction of Jacobi fields in the aforementioned papers is rather technical and complicated, especially for the higher multiplicity case in \cite{Wang19}. For the purpose of this paper, we only need generic countability instead of finiteness. We now give a more direct proof of this weaker result based on the work of \cite{GWZ18}. (Recall Section \ref{S:main} for our definition of $\Index(\Sigma)$ for almost properly embedded FBMH $\Sigma$ in $M$.)

\begin{definition}
For any $\Lambda>0$ and $I  \in \mb N$, we let $\mc M(\Lambda,I)$ be the collection of all smooth, almost properly embedded FBMH $\Sigma$ in $M$ with $\area(\Sigma) \leq \Lambda$ and $\Index(\Sigma)\leq I$.
\end{definition}

First, we show that a sequence of FBMHs in $\mc M(\Lambda,I)$ with \emph{the same index} would converge, after passing to a subsequence, to a limiting FBMH in $\mc M(\Lambda,I)$ which is either degenerate or has strictly lower index.

\begin{proposition}
\label{prop:degenerate or index decrease}
Let $\{\Sigma_j\}_{j \in \mb N}$ be a sequence of FBMHs in $\mc M(\Lambda, I)$ with $\mathrm{index}(\Sigma_j)=I$ for all $j$. Then after passing to a subsequence, $\Sigma_j$ will converge away from finitely many points locally smoothly (with multiplicity) to some $\Sigma\in \mc M(\Lambda,I)$. Moreover, either $\mathrm{index}(\Sigma)\leq I-1$ or $\Sigma$ is \emph{degenerate}, i.e. $\Sigma$ admits a non-trivial Jacobi field.
\end{proposition}

\begin{proof}
If $\Sigma_j$ smoothly converges globally to $\Sigma$ with multiplicity one, then $\Sigma$ admits a non-trivial Jacobi field by the arguments in \cite{GWZ18}. 

It remains to consider the case that $\Sigma_j$ does not globally smoothly converge to $\Sigma$. We assume that $\mathrm{index}(\Sigma)=k$. Since the convergence is not smooth, there exists $p\in\Sigma$ so that for any $r>0$, $B_r(p)\cap \Sigma_j$ does not smoothly converge to $B_r(p)\cap\Sigma$. We now prove that $\mathrm{index}(\Sigma_j)\geq k+1$ for some $j$ large, which implies that $k \leq \Index(\Sigma_j)-1=I-1$.

Let $X_1,...,X_k$ be $k$ linearly independent normal vector fields on $\Sigma$ so that they span a linear subspace on which $Q_\Sigma$ is negative-definite. By normalization, we can assume $\int_{\Sigma}|X_i|^2=1$ for $1\leq i\leq k$. Since $X_i$ vanishes along the touching set of $\Sigma_i$, we can extend each $X_i$ to a smooth vector field on $M$ which is tangential to $\partial M$. Let us still denote the extended vector field by $X_i$.

By shrinking the radius $r$ if necessary, we can assume that $\{X_i|_{\Sigma\setminus B_r(p)}\}_{i=1}^k$ is linearly independent. Let $\xi_r$ be a logarithmic cut-off function (see Appendix \ref{sec:appendix:log cutoff}) satisfying $0\leq \xi_r\leq 1$ and $\xi_r|_{B_r(p)}=0$ and $\int_\Sigma|\nabla\xi_r|^2\rightarrow 0$ and $\xi_r\rightarrow 1$ as $r\rightarrow 0$. A direct computation yields that for $1\leq i\leq k$,
\[
Q_\Sigma(\xi_rX_i,\xi_rX_i)\leq Q_{\Sigma}(X_i,X_i) + C\int_{\Sigma}|\nabla\xi_r|^2.
\]
Since $Q_\Sigma(X_i,X_i) <0$, for $r$ small enough we have for $1\leq i\leq k$
\[
Q_\Sigma (\xi_rX_i,\xi_rX_i)<0.
\]
Recall that $\delta^2 \Sigma_j (\xi_rX_i,\xi_rX_i)\rightarrow mQ_{\Sigma}(\xi_r X_i,\xi_r X_i)$ as $j\rightarrow\infty$, where $m \in \mb N$ is the multiplicity. Then for sufficiently large $j$, we have
\begin{equation}
\label{E:k-index}
\delta^2 \Sigma_j (\xi_rX_i,\xi_rX_i)<0 \qquad \text{for } 1\leq i\leq k.
\end{equation}

On the other hand, since $\Sigma_j\cap B_r(p)$ does not smoothly converge to $\Sigma\cap B_r(p)$, $\Sigma_j$ cannot be stable away from touching set in $B_r(p)$ for some $j$ by the compactness theorem in \cite[Theorem 1.4]{GWZ18}. 
 Hence for this $j$, we can find a normal vector field of $\Sigma_j$ (vanishing on touching set, so it can be extended to all of $M$ and tangential to $\partial M$) such that $\spt X \subset B_r(p)$ and 
\begin{equation}
\label{E:extra-index}
\delta^2 \Sigma_j (X,X)<0.
\end{equation}
Therefore, (\ref{E:k-index}) and (\ref{E:extra-index}) together imply that $\mathrm{index}(\Sigma_j)\geq k+1$. This finishes the proof.
\end{proof}

We can now use an inductive argument to prove the generic countability result. Recall that a Riemannian metric $g$ on the compact manifolds $M$ with non-empty boundary is said to be \emph{bumpy} if no almost properly embedded FBMH in $M$ admits a non-trivial Jacobi field.

\begin{proposition}[Generic countability]
\label{prop:countablity of W}
Let $(M,\partial M)$ be a compact manifold with boundary equipped with a bumpy metric $g$. Then $\mc M(\Lambda,I)$ is at most countable for any fixed $\Lambda,I\geq0$.
\end{proposition}
\begin{proof}
We prove it by induction. It follows from Proposition \ref{prop:degenerate or index decrease} and the bumpiness of $g$ that $\mc M(\Lambda, 0)$ is finite. To establish the induction hypothesis, suppose that $\mc M(\Lambda, I)$ is countable for some $I \geq 0$. Then for any $r>0$, $\mc M(\Lambda,I+1)\setminus \cup_{\Sigma\in \mc M(\Lambda,I)}\overline{\mf B}^{\mf F}_r(\Sigma)$ is finite by Proposition \ref{prop:degenerate or index decrease}. Since
\[\mc M(\Lambda,I+1)=\bigcup_{k=1}^{\infty}\Big(\mc M(\Lambda,I+1)\setminus \bigcup_{\Sigma\in \mc M(\Lambda,I)}\overline{\mf B}^{\mf F}_{1/k}(\Sigma)\Big),\]
we conclude that $\mc M(\Lambda, I+1)$ is countable. This completes the proof by induction.
\end{proof}

\begin{remark}
The proof of Proposition \ref{prop:countablity of W} can also be used to obtain countability of $\mathcal W^{k+1}$ in the proof of \cite[Theorem 3.6]{Zhou19} without assuming that every $\Si\in \mathcal P^h$ is properly embedded in the definition of {\em good pair} in \cite[Section 3.3]{Zhou19}.
\end{remark}


\subsection{Unstable varifolds} 

To prove the Deformation Theorem, one needs to generalize the concept of Morse index for almost properly embedded FBMHs to the context of varifolds. In what follows, we use $B^k$ to denote the open unit ball (centered at origin) in $\mathbb{R}^k$ and $\overline{\bf B}_r^{\bf F}(V)$ to denote the closed ball of radius $r>0$ centered at $V \in \mc V_n(M)$ with respect to the $\bf F$-metric. A bar above it would denote its closure (in the corresponding metric).

\begin{definition}[cf. {\cite[Definition 4.1]{MN16}}]
\label{D:k-unstable}
Let $\Sigma \in \mc V_n (M)$ be stationary in $M$ with free boundary and $\epsilon\geq 0$. We
say that $\Sigma$ is {\em $k$-unstable in an $\epsilon$-neighborhood} if there exist $0<c_0 < 1$ and a smooth $k$-parameter family $\{F_v\}_{v\in \overline{B}^k}\subset \mathrm{Diff}(M)$ with $F_0 = \mathrm{Id}$, $F_{-v} = F_v^{-1}$ for all $v\in\overline B^k$ such that, for any $V\in \overline {\bf B}_{2\epsilon}^{\bf F}(\Sigma)$, the smooth function
\[ A^V:\overline B^k\rightarrow [0, \infty), \ \ \ A^V(v)=\Vert (F_v)_{\#}V\Vert (M)\]
satisfies:
\begin{itemize}
\item  $A^V$ has a unique maximum at $m(V)\in B^k_{c_0/\sqrt{10}}$;
\item $-\frac{1}{c_0}\,\mathrm{Id}\leq D^2A^V(u)\leq -c_0\,\mathrm{Id}$  for all $u\in \overline B^k$.
\end{itemize}
Here $(F_v)_\#$ denotes the push-forward operation. Also, because $\Sigma$ is stationary in $M$ with free boundary, we necessarily have $m(\Sigma)=0$.

We say that $\Sigma$ is {\em $k$-unstable} if it is stationary and $k$-unstable in an $\epsilon$-neighborhood for some $\epsilon>0$.
\end{definition}

\begin{remark}
If $\Sigma$ is a smooth, almost properly embedded FBMH with $\Index(\Sigma) \geq k$, then $\Sigma$ is $k$-unstable in the sense of Definition \ref{D:k-unstable}  (c.f. \cite[Proposition 4.3]{MN16}). However, the converse may not be true in general (e.g. when the touching set of $\Sigma$ has a positive $\mathcal{H}^n$-measure).
\end{remark}

\begin{lemma}
\label{lemma:deform:2}
For any $\delta<1/4$, there exists $T=T(\delta,\epsilon,\Sigma,\{F_v\},c_0)\geq 0$ such that for any $V\in \overline{\mf{B}}_{2\epsilon}^{\mf{F}}(\Sigma)$ and $v\in \overline{B}^k$ with $|v-m(V)|\geq \delta$, we have 
\[A^V(\phi^V(v,T))<A^V(0)-\frac{c_0}{10},\quad\text{ and }\quad |\phi^V(v,T)|>\frac{c_0}{4}.\]
\end{lemma}

\begin{proof}
The proof is the same as that of \cite[Lemma 4.5]{MN16}.
\end{proof}

\subsection{Deformation theorem}

We now prove the key Deformation Theorem. Suppose that $X$ is a cubical complex of dimension $k$. Let $\{\Phi_i\}_{i\in \mb{N}}$ be a sequence of continuous maps from $X$ into $\mc{Z}_n(M,\partial M;\mf{F};G)$. Set
\[L=\mf{L}(\{\Phi_i\}_{i\in \mb{N}}):=\limsup_{i\to \infty}\sup_{x\in X}\mf{M}(\Phi_i(x)).\]
We will adapt the Deformation Theorem A in \cite{MN16} to our setting. Basically, the deformation theorem can produce another  sequence which is homotopic to $\{\Phi_i\}_{i\in \mb{N}}$ such that the new sequence avoids free boundary minimal hypersurfaces with large index. 

\begin{theorem}[Deformation Theorem]
\label{thm:deform}
Suppose that 
\begin{enumerate}
\item $\Sigma\in \mc{V}_n(M)$ is stationary in $M$ with free boundary and $(k+1)$-unstable;
\item $K\subset \mc{V}_n(M)$ is a subset such that $\mf{F}(\Sigma,K)>0$ and $\mf{F}(|\Phi_i|(X),K)>0$ for all $i\geq i_0$;
\item $\|\Sigma\|(M)=L$.
\end{enumerate}
Then there exist $\bar{\epsilon}>0$, $j_0\in \mb{N}$, and another sequence $\{\Psi_i\}_{i\in \mb{N}}$ of maps from $X$ into  $\mc{Z}_n(M,\partial M;\mf{F};G)$ so that 
\begin{enumerate}[label=(\roman*)]
\item $\Psi_i$ is homotopic to $\Phi_i$ in the $\mf{F}$-topology for all $i\in \mb{N}$;
\item $\mf{L}(\{\Psi_i\}_{i\in \mb{N}})\leq L$;
\item  $\mf{F}(|\Psi_i|(X), \overline{\mf{B}}_{\bar{\epsilon}}^{\mf{F}}(\Sigma)\cup K)>0$ for all $i\geq j_0$.
\end{enumerate}
\end{theorem}

\begin{remark}
Note that the subset $K$ may not be compact in our applications.
\end{remark}

\begin{proof}[Proof of Theorem \ref{thm:deform}]
Since the strategy of the proof follows from those of \cite[Theorem 5.1]{MN16}, we will only sketch the main steps and point out necessary modifications for the free boundary setting. 

Let $d=\mf{F}(\Sigma,K)>0$. By assumption (1), there exists a constant $\epsilon>0$ such that $\Sigma$ is $(k+1)$-unstable in an $\epsilon$-neighborhood. Suppose that $\{F_v\}_{v\in \oli{B}^{k+1}}$ and $c_0$ are given as in Definition \ref{D:k-unstable}. Since $\mf F(\Sigma,K)>0$, by possibly changing $\{\epsilon, \{F_v\},c_0\}$, we may assume that  for any $V\in \oli{\mf{B}}_{2\epsilon}^\mf{F}(\Sigma)$,
\begin{equation}
\label{equ:deform:3}
\min_{v\in \oli{B}^{k+1}} \mf F((F_v)_{\#}V,K) > \frac{d}{2}. 
\end{equation}

For each fixed $i\in \mb{N}$, since $\Phi_i:X\to \mc{Z}_n(M,\partial M; \mf{F};G)$ is  continuous, we may assume that $X(k_i)$ is a sufficiently fine subdivision of $X$ such that 
\[\mf{F}(|\Phi_i(x)|,|\Phi_i(y)|)<\delta_i\]
for any $x,y$ belonging to the same cell in $X(k_i)$ with $\delta_i=\min\{2^{-(i+k+2)},\epsilon/4\}$. Recall that for $\tau \in \mc{Z}_n(M,\partial M; \mf{F};G)$, $|\tau|$ is defined to be $|T|$ where $T \in Z_n(M,\partial M;G)$ is the unique canonical representative of $\tau$.

Note that  $A^V: \oli{B}^{k+1}\to [0,\infty)$ is a smooth function for any $V\in  \oli{\mf{B}}_{2\epsilon}^\mf{F}(\Sigma)$, so we can assume that for any $x,y$ belonging to the same cell in $X(k_i)$ with $\mf{F}(|\Phi_i(x)|,\Sigma)\leq 2\epsilon$ and  $\mf{F}(|\Phi_i(y)|,\Sigma)\leq 2\epsilon$, we have
\[|m(|\Phi_i(x)|) -m(|\Phi_i(y)|)|<\delta_i.\]

For any $\eta>0$, we use $U_{i,\eta}$ to denote the union of all cells $\sigma\in X(k_i)$ so that $\mf{F}(|\Phi_i(x)|,\Sigma)<\eta$ for all $x\in \sigma$. Then $U_{i,\eta}$ is a subcomplex of $X(k_i)$. If a cell $\beta\notin U_{i,\eta}$, then there exists some point $x'\in \beta$ such that $\mf{F}(|\Phi_i(x')|,\Sigma)\geq \eta$. Hence, for any $y\in \beta$, we have 
\begin{equation}\label{equ:deform:notinu}
\mf{F}(|\Phi_i(y)|,\Sigma)\geq \eta-\delta_i.
\end{equation}
In the following, for any $x\in U_{i,2\epsilon}$, we will use the notation 
\[A_i^x=A^{|\Phi_i(x)|},\,\, m_i(x)=m(|\Phi_i(x)|)\,\text{ and }\, \phi_i^x=\phi^{|\Phi_i(x)|}.\]
Following the construction in \cite[Theorem 5.1]{MN16}, we can construct a continuous homotopy 
\[\hat{H}_i:U_{i,2\epsilon}\times [0,1]\to B_{2^{-i}}^{k+1}(0)\quad \text{so that } \hat{H}_i(x,0)=0\,\, \forall\, x\in U_{i,2\epsilon},\]
and 
\begin{equation}\label{equ:deform:4}
\inf_{x\in U_{i,2\epsilon}}|m_i(x)-\hat{H}_i(x,1)|\geq \eta_i>0\quad \text{for some}\, \eta_i>0.
\end{equation}
The key idea is that the subspaces \[A_i=\{(x,0)\in X\times \oli{B}^{k+1}:  x\in U_{i,2\epsilon}\}\quad \text{and}\quad B_i=\{(x,m_i(x))\in X\times \oli{B}^{k+1} :x\in U_{i,2\epsilon}\}\] have both dimension at most $k$ and are contained in a space of dimension $2k+1$. So it is possible to perturb $A_i$ slightly such that $A_i\cap B_i=\emptyset$ which gives (\ref{equ:deform:4}).

Let $c:[0,\infty)\to [0,1]$ be a cutoff function which is non-increasing, and $c$ is equal to 1 in a neighborhood of $[0,3\epsilon/2]$, and $0$ in a neighborhood of $[7\epsilon/4,\infty)$. By (\ref{equ:deform:notinu}), if $y\notin U_{i,2\epsilon}$, then 
$\mf{F}(|\Phi_i(y)|,\Sigma)\geq 2\epsilon-\delta_i\geq 7\epsilon/4$, since $\delta_i\leq \epsilon/4.$ Therefore, 
\[c(\mf{F}(|\Phi_i(y)|,\Sigma))=0\,\,\text{ for any }\, y\notin U_{i,2\epsilon}.\]
We now consider a map $H_i:X\times [0,1]\to B_{2^{-i}}^{k+1}(0)$ given by 
\[H_i(x,t)=\hat{H}_i(x,c(\mf{F}(|\Phi_i(x)|,\Sigma))t) \quad \text{ if }\, x\in U_{i,2\epsilon}\]
and 
\[H_i(x,t)=0 \quad \text{ if }\, x\in X\setminus U_{i,2\epsilon}.\]
Then $H_i$ is continuous.

Now we are ready to construct the new sequence $\{\Psi_i\}$ of maps from $X$ into  $\mc{Z}_n(M,\partial M;\mf{F};G)$. With $\eta_i$ given by (\ref{equ:deform:4}), let $T_i=T_i(\eta_i,\epsilon,\Sigma,\{F_v\},c_0)$ be given by Lemma \ref{lemma:deform:2}. Set 
$D_i:X\to \oli{B}^{k+1}$ such that
\[D_i(x)=\phi_i^x\big(H_i(x,1), c(\mf{F}(|\Phi_i(x)|,\Sigma))T_i\big)\quad \text{ if }\, x\in U_{i,2\epsilon}\]
and 
\[D_i(x)=0\quad \text{ if } x\in X\setminus U_{i,2\epsilon}.\] 
Then $D_i$ is continuous. Now we define
\[\Psi_i:X\to \mc{Z}_n(M,\partial M;\mf{F};G),\quad \Psi_i(x)=(F_{D_i(x)})_{\#}(\Phi_i(x)).\]
In particular, we have
\[\Psi_i(x)=\Phi_i(x),\quad \text{ if }\, x\in X\setminus U_{i,2\epsilon}.\]

Since the map $D_i$ is homotopic to the zero map in $\oli{B}^{k+1}$, we obtain that $\Psi_i$ is homotopic to $\Phi_i$ in the $\mf{F}$-topology for all $i\in \mb{N}$. Following the steps in \cite[Theorem 5.1]{MN16}, one can check $\Psi$ is the desired map.

\end{proof}

\section{Index estimates}
\label{sec:index estimate}

In this section, we will use the deformation theorem (Theorem \ref{thm:deform}) to prove the index estimates. We will first prove such estimates for manifolds with \emph{bumpy} metrics. Recall that a metric is bumpy if every smooth almost properly embedded FBMH is non-degenerate, i.e. admits no non-trivial Jacobi field (which can be non-zero on the touching set). It was proved by White (see \cite{Whi91}, \cite{Whi17}) and Ambrozio-Carlotto-Sharp \cite{ACS17} (in the free boundary setting) that bumpy metrics are generic in the Baire sense. 


\begin{theorem}[Index estimates for bumpy metrics]
\label{thm:index:bumpy}
Let $(M^{n+1},\partial M)$ be a compact manifold with boundary equipped with a bumpy metric $g$ and $3\leq (n+1)\leq 7$. Let $X$ be a $k$-dimensional cubical complex and $\Phi: X\to \mc{Z}_n(M,\partial M; \mf{F}; G)$ be a continuous map. Let $\Pi$ be the class of all continuous maps $\Phi':X\to \mc{Z}_n(M,\partial M;\mf{F};G)$ such that $\Phi$ and $\Phi'$ are homotopic to each other in the flat topology.  Suppose that $\{\Phi_i\}_{i\in \mb{N}}$ is a min-max sequence in $\Pi$  such that 
\[L=\mf{L}(\{\Phi_i\}_{i\in \mb{N}})=\mf{L}(\Pi)>0.\]

Then there is $\Sigma\in \mf{C}(\{\Phi_i\}_{i\in \mb{N}})$ with support a smooth, compact, almost properly embedded, free boundary minimal hypersurface such that 
\[\mf{L}(\Pi)=\|\Sigma\|(M)\quad\text{ and }\,\,\Index(\spt \Sigma)\leq k.\]
\end{theorem}
\begin{proof}
By the compactness result \cite[Theorem 1.1]{GWZ18} for free boundary minimal hypersurfaces, it suffices to show that for any $r>0$, there exists a  varifold $\wti{\Sigma}\in \mc{V}_n(M)$ which is stationary in $M$ with free boundary and  whose support is a smooth compact embedded minimal hypersurface such that  $\mf{F}\big(\wti{\Sigma}, \mf{C}(\{\Phi_i\}_{i\in \mb{N}})\big)<r,$
\[\mf{L}(\Pi)=\|\wti{\Sigma}\|(M),\quad \text{ and  }\,\, \Index(\spt \wti{\Sigma})\leq k.\]
Once we have this, we can choose $\wti{\Sigma}_j \in \mc{V}_n(M)$ such that $\mf{F}\big(\wti{\Sigma}_j, \mf{C}(\{\Phi_i\})\big)<j^{-1}$ and thus the varifold limit $\wti{\Sigma}$ of $\wti{\Sigma}_j$ satisfies $\wti \Sigma \in \mf{C}(\{\Phi_i\})$.

Let $\mc{W}$ be the set of all stationary varifolds $V$ in $M$ with free boundary such that $\|V\|(M)=L$ and the support of $V$ is a smooth compact embedded free boundary minimal hypersurface. Now we fix $r>0$ and set 
\[\mc{W}(r):=\{V\in \mc{W}: \mf{F}\big(V, \mf{C}(\{\Phi_i\}_i)\big)\geq r\}.\]
We can easily argue by contradiction to obtain the following result.
\begin{lemma}\label{lemma:index:empty}
There exists $i_0\in \mb{N}$ and $\epsilon_0>0$ such that $\mf{F}(|\Phi_i|(X),\mc{W}(r))>\epsilon_0$ for all $i\geq i_0$.
\end{lemma}

Let $\mc{W}^{k+1}$ be the collection of elements in $\mc{W}$ whose support has index greater than or equal to $(k+1)$. Now it suffices to show $\mc{W}\setminus (\mc{W}(r)\cup \mc{W}^{k+1})$ is non-empty.

Since the metric $g$ is bumpy, the set $\mc{W}^{k+1}$ is countable by Proposition \ref{prop:countablity of W}. So we can write 
\[\mc{W}^{k+1}\setminus \overline{\mf{B}}_{\epsilon_0}^{\mf{F}}(\mc{W}(r))=\{\Sigma_1,\Sigma_2,\Sigma_3,\ldots\}.\]
Note that for any $i\in \mb{N}$,  $\Sigma_i$ satisfies $\Index(\spt \Sigma_i)\geq (k+1)$ and thus $\Sigma_i$ is $(k+1)$-unstable. 

Then our argument from here is the same as that of \cite{MN16}, and we also present it here for the sake of completeness.

Using our Deformation Theorem (Theorem \ref{thm:deform}) with $K=\overline{\mf{B}}_{\epsilon_0}^{\mf{F}}(\mc{W}(r))$ and $\Sigma=\Sigma_1$ (recall that $\mf{F}(|\Phi_i|(X), K)>0$ for all $i\geq i_0$ by Lemma \ref{lemma:index:empty}), we can find $\epsilon_1>0$, $i_1\in \mb{N}$, and $\{\Phi_i^1\}_{i\in \mb{N}}$ so that 
\begin{itemize}
\item $\Phi_i^1$ is homotopic to $\Phi_i$ in the $\mf{F}$-topology for all $i\in \mb{N}$;
\item $\mf{L}(\{\Phi_i^1\}_i)\leq L$;
\item  $\mf{F}(|\Phi_i^1|(X), \overline{\mf{B}}_{\epsilon_1}^{\mf{F}}(\Sigma_1)\cup \overline{\mf{B}}_{\epsilon_0}^{\mf{F}}(\mc{W}(r)))>0$ for all $i\geq i_1$;
\item  no $\Sigma_j$ belongs to $\partial \overline{\mf{B}}_{\epsilon_1}^{\mf{F}}(\Sigma_1)$.(This can be easily satisfied since $\{\Sigma_1,\Sigma_2,\cdots\}$ is a countable set.)
\end{itemize}

Next, we consider $\Sigma_2$. If $\Sigma_2\notin \overline{\mf{B}}_{\epsilon_1}^{\mf{F}}(\Sigma_1)$, then we apply Theorem \ref{thm:deform} again with $K=\overline{\mf{B}}_{\epsilon_1}^{\mf{F}}(\Sigma_1)\cup \overline{\mf{B}}_{\epsilon_0}^{\mf{F}}(\mc{W}(r))$ and find $\epsilon_2$, $i_2\in \mb{N}$, and $\{\Phi_i^2\}_{i\in \mb{N}}$ so that 
\begin{itemize}
\item $\Phi_i^2$ is homotopic to $\Phi_i$ in the $\mf{F}$-topology for all $i\in \mb{N}$;
\item $\mf{L}(\{\Phi_i^2\}_i)\leq L$;
\item  $\mf{F}(|\Phi_i^2|(X), \overline{\mf{B}}_{\epsilon_2}^{\mf{F}}(\Sigma_2)\cup \overline{\mf{B}}_{\epsilon_1}^{\mf{F}}(\Sigma_1)\cup \overline{\mf{B}}_{\epsilon_0}^{\mf{F}}(\mc{W}(r)))>0$ for all $i\geq i_2$;
\item  no $\Sigma_j$ belongs to $\partial \overline{\mf{B}}_{\epsilon_1}^{\mf{F}}(\Sigma_1)\cup \partial \overline{\mf{B}}_{\epsilon_2}^{\mf{F}}(\Sigma_2)$. 
\end{itemize}
If $\Sigma_2\in \mf{B}_{\epsilon_1}^{\mf{F}}(\Sigma_1)$, we skip $\Sigma_2$ and repeat the procedure with $\Sigma_3$. 

We keep applying the above procedure and eventually there are two possibilities. The first case is that we can find for all $l\in \mb{N}$, there exists a sequence $\{\Phi_i^l\}_{i\in \mb{N}}$, $\epsilon_l$, $i_l\in \mb{N}$, and $\Sigma_{j_l}\in \mc{W}^{k+1}\setminus  \overline{\mf{B}}_{\epsilon_0}^{\mf{F}}(\mc{W}(r))$ for some $j_l\in \mb{N}$ so that 
\begin{enumerate}[label=(\roman*)]
\item $\Phi_i^l$ is homotopic to $\Phi_i$ in the $\mf{F}$-topology for all $i\in \mb{N}$;
\item $\mf{L}(\{\Phi_i^l\}_i)\leq L$;
\item  $\mf{F}(|\Phi_i^l|(X), \cup_{q=1}^l \overline{\mf{B}}_{\epsilon_q}^{\mf{F}}(\Sigma_{j_q})\cup  \overline{\mf{B}}_{\epsilon_0}^{\mf{F}}(\mc{W}(r)))>0$ for all $i\geq i_l$;
\item $\{\Sigma_1,\ldots,\Sigma_l\}\subset \cup_{q=1}^l \mf{B}_{\epsilon_q}^{\mf{F}}(\Sigma_{j_q}) $;
\item  no $\Sigma_j$ belongs to $\partial \overline{\mf{B}}_{\epsilon_1}^{\mf{F}}(\Sigma_1)\cup \cdots\cup  \partial \overline{\mf{B}}_{\epsilon_l}^{\mf{F}}(\Sigma_{j_l})$. 
\end{enumerate}
The second case is that the process stops in finitely many steps. This means that we can find some $m\in \mb{N}$, a sequence $\{\Phi_i^m\}_{i\in \mb{N}}$, $\epsilon_1,\ldots,\epsilon_m>0$, $i_m\in \mb{N}$, and $\Sigma_{j_1},\ldots,\Sigma_{j_m}\in \mc{W}^{k+1}\setminus  \overline{\mf{B}}_{\epsilon_0}^{\mf{F}}(\mc{W}(r))$ so that 
\begin{enumerate}[label=(\alph*)]
\item $\Phi_i^m$ is homotopic to $\Phi_i$ in the $\mf{F}$-topology for all $i\in \mb{N}$;
\item $\mf{L}(\{\Phi_i^m\}_i)\leq L$;
\item $\mf{F}(|\Phi_i^m|(X), \cup_{q=1}^m \overline{\mf{B}}_{\epsilon_q}^{\mf{F}}(\Sigma_{j_q})\cup  \overline{\mf{B}}_{\epsilon_0}^{\mf{F}}(\mc{W}(r)))>0$ for all $i\geq i_m$;
\item $\{\Sigma_j:j\geq 1\}\subset \cup_{q=1}^m \mf{B}_{\epsilon_q}^{\mf{F}}(\Sigma_{j_q}) $.
\end{enumerate}

For either case, we will choose a min-max sequence so that we can apply the Min-max Theorem (Theorem \ref{thm:minmax}). For the first case, we can choose a diagonal sequence $\{\Phi_{p_l}^l\}_{l\in \mb{N}}$ and set $\Psi_l=\Phi_{p_l}^l$, where $\{p_l\}_{l\in \mb{N}}$ is an increasing sequence such that $p_l\geq i_l$ (the condition (iii) is satisfied) and \[\sup_{x\in X}\|\Phi_{p_l}^l (x)\|(M)\leq L+\frac{1}{l}.\]
For the second case, we simply choose the last sequence $\{\Phi_l^m\}_l$  and set $p_l=l$ and $\Psi_l=\Phi_l^m$. Now it is easy to see that for both cases, the new sequence $\{\Psi_l\}_{l\in \mb{N}}$ satisfies the following conditions:
\begin{enumerate}
\item $\Psi_l$ is homotopic to $\Phi_{p_l}$ in the $\mf{F}$-topology for all $l\in \mb{N}$;
\item $\mf{L}(\{\Psi_l\}_l)\leq L$;
\item $\mf{C}(\{\Psi_l\}_l)\cap \big(\mc{W}^{k+1}\cup \mc{W}(r)\big)=\emptyset$.
\end{enumerate}
Note that the condition (3) follows directly from (iii) or (c). Then Theorem \ref{thm:minmax} implies that there exists a  varifold  $V\in \mf{C}(\{\Psi_l\}_l)$ such that $V$ is stationary in $M$ with free boundary, and  $V$ is supported on a smooth compact embedded free boundary minimal hypersurface. By (3), we know that $\mc{W}\setminus \big(\mc{W}^{k+1}\cup \mc{W}(r)\big)$ is non-empty and this completes the proof of index estimates for bumpy metrics.
\end{proof}

Now we are ready to prove the general index estimates using Theorem \ref{thm:index:bumpy} and the compactness result of \cite{GWZ18}. 

\begin{theorem}[Index estimates for general metrics]
\label{thm:index estimates}
Suppose that $(M^{n+1},g)$ is a smooth compact manifold with boundary and $3\leq (n+1)\leq 7$. Let $X$ be a cubical complex of dimensional $k$ and $\Phi:X\to \mc{Z}_n(M,\partial M;\mf{F};G)$ be a continuous map. Let $\Pi$  denote the associated homotopy class of $\Phi$.  Then there exists a varifold $V\in \mc{V}_n(M)$ such that 
\begin{enumerate}[label=(\roman*)]
\item $\|V\|(M)=\mf{L}(\Pi)$;
\item $V$ is stationary in $M$ with free boundary;
\item there exists $N\in \mb{N}$ and $m_i\in \mb{N}$, $1\leq i\leq N$, such that $V=\sum_{i=1}^N m_i|\Sigma_i|$, where each $\Sigma_i$ is a smooth, compact, connected, almost properly embedded, free boundary minimal hypersurface in $M$. Moreover, 
\[ \Index(\spt V)=\sum_{i=1}^N \Index(\Sigma_i)\leq k.\]
\end{enumerate}
\end{theorem}
\begin{proof}
Since bumpy metrics are generic in the Baire sense \cite{ACS17}, we can take a sequence $\{g_j\}_{j\in \mb{N}}$ of bumpy metrics converging smoothly to $g$. For each $g_j$, we use $L_j$ to denote the width of $\Pi$ with respect to $g_j$. By Theorem \ref{thm:index:bumpy}, we know that there exists a varifold $V_j\in \mc{V}_n(M)$ which is stationary in $M$ with free boundary and whose support is a smooth, compact, almost properly embedded, free boundary minimal hypersurface. Moreover, 
\[L_j=\|V_j\|(M)\quad \text{and} \quad \Index(\spt V_j)\leq k.\]   
Since the width is continuous with respect to metrics, we know that $L_j\to \mf{L}(\Pi)$ as $j\to \infty$. The conclusion then follows directly from the compactness theorem in \cite{GWZ18}. 
\end{proof}

\section{Density of free boundary minimal hypersurfaces}
\label{sec:density}

In the final section, we give a proof of the density result (Theorem \ref{T:density}). We shall need the notions of $p$-widths for compact Riemannian manifolds (with or without boundary). We will first recall the definitions and state the relevant results in the free boundary setting. 

\subsection{Width}

Let $X$ denote a cubical subcomplex of the $m$-dimensional cube $I^m=[0,1]^m$. 

\begin{definition}
\label{D:sweepout}
Given $p\in \mb N$, a continuous map in
the flat topology
\[\Phi : X \rightarrow \mc Z_{n} (M, \partial M;\mb Z_2)\]
is called a $p$-sweepout if the $p$-th cup power of $\lambda = \Phi^*(\bar{\lambda})$ is non-zero in $H^{p}(X; \mb Z_2 )$ where $0 \neq \bar{\lambda} \in H^1(\mc Z_{n} (M, \partial M;\mb Z_2);\mb Z_2) \cong \mb Z_2$. We denote by $\mc P_p(M)$ the set of all $p$-sweepouts that are continuous in the flat topology and have no concentration of mass (\cite[\S 3.7]{MN17}).
\end{definition}

\begin{definition}
The {\em $p$-width} of a Riemannian manifold $(M,g)$ with boundary is defined by
\[ \omega_p(M,g):=\inf_{\Phi\in\mc P_p (M)}\sup\{\mf M(\Phi(x)) :x\in  \mathrm{dmn}(\Phi)\},\]
where $\mathrm{dmn}(\Phi)$ is the domain of $\Phi$.
\end{definition}

The following proposition says that the $p$-width $\omega_p$ is realized by the area (counting multiplicities) of min-max free boundary minimal hypersurfaces, which is an application of our Min-max Theorem (Theorem \ref{thm:minmax}) and general Morse index estimates (Theorem \ref{thm:index estimates}), together with the compactness result of \cite[Theorem 1.1]{GWZ18}.

\begin{proposition}[cf.{\cite[Proposition 2.2]{IMN17}}]
\label{P:width}
Suppose $3\leq (n + 1)\leq 7$. Then for each $k\in\mb N$, there exist a finite disjoint collection $\{\Sigma_1,...,\Sigma_N\}$ of smooth, compact, almost properly embedded FBMHs in $(M,\partial M;g)$, and integers $\{m_1,...,m_N\}\subset\mb N$, such that
\[\omega_k(M,g)=\sum_{j=1}^Nm_j\area_g(\Sigma_j)\ \ \text{ and }\ \ \sum_{j=1}^N \mathrm{index}(\Sigma_j)\leq k.\]
\end{proposition}
\begin{proof}
Choose a sequence $\{\Phi_i\}_{i\in\mb N}\subset \mc P_k(M)$ such that
\begin{equation}\label{eq:sequence of sweepouts realize width}
\lim_{i\rightarrow\infty} \sup\{\mf M(\Phi_i(x)):x\in X_i = \mathrm{dmn}(\Phi_i)\}=\omega_k(M, g).
\end{equation}
Without loss of generality, we can assume that the dimension of $X_i$ is $k$ for all $i$ (see \cite[\S 1.5]{MN16} or \cite[Proof of Proposition 2.2]{IMN17}). 

By the Discretization Theorem \cite[Theorem 4.12]{LZ16} and the Interpolation Theorem (Theorem \ref{thm:interpolation}), we can assume that $\Phi_i$ is a continuous map to $\mc Z_n(M,\partial M;\mb Z_2)$ in the $\mf F$-metric. Denote by $\Pi_i$ the homotopy class of $\Phi_i$. This is the class of all maps $\Phi'_ i: X_i\rightarrow\mc Z_n (M,\partial M;\mb Z_2)$, continuous in the $\mf F$-metric, that are homotopic to $\Phi_i$ in the flat topology.  In particular, ${\Phi'_i}^*(\bar\lambda)=\Phi_i^*(\bar\lambda)$. 
Continuity in the $\mf F$-metric implies no concentration of mass (see Lemma \ref{lemma:no mass}), hence every such $\Phi'_i$ is also a $k$-sweepout in the sense of Definition \ref{D:sweepout}.

\begin{claim}
\label{claim:wk achieved}
$\lim_{i\rightarrow \infty}\mf L(\Pi_i)=\omega_k(M,g)$.
\end{claim}
\begin{proof}[Proof of Claim \ref{claim:wk achieved}]
Note that
\begin{equation}\label{eq:upper bound of LPi}
\mf L(\Pi_i)\leq \sup\{\mf M(\Phi_i(x)):x\in X_i\}.
\end{equation}
Letting $i\rightarrow\infty$, the right hand side tends to $\omega_k(M,g)$ by (\ref{eq:sequence of sweepouts realize width}). On the other hand, since that each element in $\Pi_i$ is also a $k$-sweepout, then
\[\omega_k(M,g)\leq \inf_{\Phi'\in\Pi_i} \sup\{\mf M(\Phi'(x)):x\in X_i\}=\mf L(\Pi_i).\]
Together with (\ref{eq:upper bound of LPi}), the desired result follows.
\end{proof}

To proceed with the arguments, Theorem \ref{thm:index estimates} implies the existence of a finite disjoint collection $\{\Sigma_{i,1} ,...,\Sigma_{i,N_i}\}$ of almost properly embedded FBMHs in $(M,\partial M)$, and a sequence of integers $\{m_{i,1},...,m_{i,N_i} \} \subset \mb N$, such that
\[ \mf L(\Pi_i )= \sum_{j=1}^{N_i}m_{i,j}\area_g(\Sigma_{i,j}),\ \ \text{ and }\ \ \sum_{j=1}^{N_i}\mathrm{index}(\Sigma_{i,j})\leq k.\]

Note that the areas of non-trivial almost properly embedded FBMHs in $(M,\partial M;g)$ are uniformly bound away from zero. Hence the number of components $N_i$ and the multiplicities $m_{i,j}$ are uniformly bounded from above. The desired results then follow immediately from the compactness theorem \cite[Theorem 1.1]{GWZ18}.
\end{proof}

Since the two formulations of min-max theory for manifolds with boundary were shown to be equivalent in Section \ref{sec:equivalence}, our $p$-widths defined using the space $\mc Z_{n} (M, \partial M;\mb Z_2)$ also satisfy a Weyl Law as in \cite{LMN16}.

\begin{theorem}[Weyl Law for the Volume Spectrum; \cite{LMN16}]
\label{thm:weyl}
There exists a constant $\alpha(n)$ such that, for every compact Riemannian manifold $(M^{n+1},g)$ with (possibly empty) boundary, we have 
\[\lim_{k\to \infty}\omega_k(M,g)k^{-\frac{1}{n+1}}=\alpha(n)\vol(M,g)^{\frac{n}{n+1}}.\]
\end{theorem}

It is known that the normalized $p$-width is a locally Lipschitz function of the metric, with a uniform local Lipschitz constant independent of $p$.

\begin{lemma}[{\cite[Lemma 2.1]{IMN17},\cite[Lemma 1]{MNS17}}]
Let $g_0$ be a $C^2$ Riemannian metric on $(M,\partial M)$, and let $C_1 < C_2$ be positive constants. Then there exists $K=K( g_0,C_1,C_2) > 0$ such that
\[
|p^{-\frac{1}{n+1}}\omega_p (M,g) -p^{-\frac{1}{n+1}}\omega_p (M, g')|\leq K \cdot |g- g'|_{g_0} \]
for all $C^2$ metrics $g, g'$ such that $C_1g_0\leq g \leq C_2 g_0$, $C_1g_0\leq g' \leq C_2 g_0$ and any $p\in\mb N$.
\end{lemma}

\subsection{Perturbation results}

Given a Riemannian manifold $(M,\partial M;g)$ with boundary and an almost properly embedded free boundary minimal hypersurface $\Sigma \subset M$, in general $\Sigma$ could be degenerate and improper. We first prove that under a smooth perturbation of $g$, we can make $\Sigma$ non-degenerate.

\begin{proposition}(cf.\cite[Proposition 2.3]{IMN17}, \cite[Lemma 4]{MNS17})
\label{prop:free:non-degenerate}
Let $\Sigma$ be a compact, smooth almost properly embedded FBMH in $(M,\partial M;g)$. Then there exists a sequence of metrics $g_i$ on $M$, $i \in \mb N$, converging to $g$ smoothly such that $\Sigma$ is a non-degenerate, almost properly embedded FBMH in $(M,\partial M;g_i)$ for each $i \in \mb N$.
\end{proposition}
\begin{proof}
If $\Sigma$ is equal to the union of some components of $\partial M$, the result follows from \cite[Proposition 2.3]{IMN17}. Otherwise, the points $\{x_i\}$ in \cite[Lemma 4]{MNS17} can be chosen to lie in the interior of $M$. Therefore, $\Sigma$ is still a FBMH under the locally conformally perturbed metrics. Finally, using the arguments in \cite[Lemma 4]{MNS17}, the perturbations of $g$ therein make $\Sigma$ non-degenerate. 
\end{proof}

Next, we prove that under suitable perturbation of the metrics \emph{and} the hypersurface $\Sigma$, we can make the FBMHs properly embedded.

\begin{proposition}
\label{prop:perturb to be proper}
Let $\Sigma$ be a compact, smooth almost properly embedded FBMH in $(M,\partial M;g)$. 
Then there exist a sequence of metrics $g_i$ on $M$, and a sequence of hypersurfaces $\Sigma_i$ in $M$ so that 
\begin{itemize}
\item $g_i$ converges to $g$ in the smooth topology;
\item $\Sigma_i$ smoothly converges to $\Sigma$;
\item $\Sigma_i$ is a properly embedded FBMH in $(M,\partial M;g_i)$ for each $i$.
\end{itemize}
\end{proposition}

\begin{proof}
Recall that $(M,g)$ can be regarded as a domain of a closed Riemannian manifold $(\wti M,\wti g)$ of the same dimension. Let $d$ be the signed distance to $\partial M$ in $\wti M$ so that $\nabla d|_{\partial M}$ is the unit normal vector field on $\partial M$ pointing out of $M$. Let $V$ be an open set of $\wti M$ so that $\overline V\cap \partial \Sigma=\emptyset$ and $\mathrm{Touch}(\Sigma)\subset V$. Let $\xi$ be a nonnegative cut-off function supported in $V$ such that $\xi(x)=1$ for all $x\in\mathrm{Touch}(\Sigma)$ and $\xi \nabla d$ is a smooth vector field on $\wti M$. Denote by $\{F_t\}_{t\in[0,1]}$ the family of diffeomorphisms of $\wti M$ generated by $\xi\nabla d$.

Now let $g_i=F^*_{1/i}\wti g$ and $\Sigma_i=(F_{1/i})^{-1}(\Sigma)$. Since $F_t$ is a diffeomorphism of $\wti M$, $g_i\rightarrow g$ and $\Sigma_i\rightarrow\Sigma$ smoothly. Note that $\Sigma$ is a properly embedded FBMH in $(F_t(M),\wti g)$. In other words, $\Sigma_i$ is a properly embedded FBMH in $M$ with respect to the metric $g_i$. This completes the proof.
\end{proof}

\subsection{Proof of Theorem \ref{T:density}} 

As in the proof of \cite[Main theorem]{IMN17}, Theorem \ref{T:density} follows once we have proved the following proposition.

\begin{proposition}
\label{conj:free:dense}
Let  $(M^{n+1},\partial M;g)$ be a compact Riemannian manifold with boundary and $3 \leq (n+1)\leq 7$. Let $\mc{M}$ be the space of all smooth Riemannian metrics on $M$, endowed with the smooth topology. Suppose that $U\subset M$ is a non-empty relatively open subset.  Let $\mc{M}_U$ be the set of  metrics $g\in \mc{M}$ such that there exists a non-degenerate, properly embedded FBMH $\Sigma$ in $(M,g)$ which intersects $U$. Then $\mc{M}_U$ is open and dense in $\mc{M}$ in the smooth topology.
\end{proposition}
\begin{proof}
Let $g \in\mc M_U$ and $\Sigma$ be as in the statement of the proposition. Because $\Sigma$ is properly embedded and non-degenerate, from the Structure Theorem of
White \cite[Theorem 2.1]{Whi91} (see \cite[Theorem 35]{ACS17} for a version in the free boundary setting), for every Riemannian metric $g'$ sufficiently close to $g$, there
exists a unique non-degenerate properly embedded FBMH $\Sigma'$ close to $\Sigma$. This implies $\mc M_U$ is open.

It remains to show the set $\mc M_U$ is dense. Let $g$ be an arbitrary smooth Riemannian metric on $M$ and $\mc V$ be an arbitrary neighborhood of $g$ in the
$C^\infty$ topology. By the Bumpy Metrics Theorem (\cite[Theorem 2.1]{Whi91},\cite[Theorem 9]{ACS17}), there exists $g'\in\mc V$ such that every compact, almost properly embedded FBMH with respect to $g'$ is non-degenerate. If one of these hypersurfaces is almost properly embedded and intersects $U$, then by Proposition \ref{prop:perturb to be proper}, there exist a sequence of metrics $g_i$ on $M$, and a sequence of hypersurfaces $\Sigma_i$ so that 
\begin{itemize}
\item $g_i$ converges to $g'$ in the smooth topology;
\item $\Sigma_i$ smoothly converges to $\Sigma$;
\item $\Sigma_i$ is a properly embedded FBMH in $(M,\partial M;g_i)$ for each $i \in \mb N$.
\end{itemize}
Then for $i$ large enough, $g_i\in\mc V$ and $\Sigma_i$ is a properly embedded FBMH with respect to $g_i$ so that $\Sigma_i\cap U\neq \emptyset$. This implies that $g_i\in \mc M_U$ and we are done.

Hence we can suppose that every almost properly embedded FBMH with respect to $g'$ is contained in the complement of $U$. Since $g'$ is bumpy, it follows from Proposition \ref{prop:countablity of W} $\mc M(\Lambda, I)$ is countable with respect to $g'$ for any $\Lambda>0$ and $I\in \mb{N}$. Therefore, the set
\[
\mc C:=
\left\{\sum_{j=1}^Nm_j\area_{g'}(\Sigma_j)\,\Big|
\begin{array}{ll}
&N\in\mb N, \{m_j\}\subset \mb N,\{\Sigma_j\} \text{ disjoint collection of almost}\\
&\text{ properly embedded FBMHs in $(M,\partial M; g')$}
\end{array} 
\right\}
\]
is countable. Following the arguments in \cite[Proposition 3.1]{IMN17} and using the Weyl Law (Theorem \ref{thm:weyl}) and Proposition \ref{P:width}, there exists $g''\in\mc V$ so that $(M,\partial M;g'')$ admits an almost properly embedded FBMH intersecting $U$. Then by Proposition \ref{prop:free:non-degenerate} and \ref{prop:perturb to be proper}, we have $\mc V\cap \mc M_U\neq \emptyset$.
\end{proof}

\appendix
\section{logarithmic cut-off trick}\label{sec:appendix:log cutoff}

We recall a construction of the logarithmic cut-off functions used in this paper.

\begin{lemma}
Let $(M^{n+1},g)$ be a Riemannian manifold with $n+1\geq 3$ and $\Sigma$ be a hypersurface in $M$. For $p\in \Sigma$, there exists a family of cut-off functions $\{\xi_r\}$ on $M$ satisfying
\begin{enumerate}
\item $\xi_r|_{B_r(p)}=0$ and $0\leq \xi_r(x)\leq 1$ for $x\in M$;
\item $\int_\Sigma|\nabla\xi_r|^2\rightarrow 0$ and $\xi_r\rightarrow 1$ as $r\rightarrow 0$.
\end{enumerate}
\end{lemma}
\begin{proof}
For simplicity, we denote $|x|:=\dist(x,p)$. Now define $\xi_r(x)$ as follows:
\[
\xi_r(x):=
\left\{
\begin{aligned}
& 0 ,&|x|\leq r;\\
& 2-\frac{2\log|x|}{\log r}, &r< |x|\leq \sqrt r;\\
& 1, &|x|>\sqrt r.
\end{aligned} 
\right.
\]
One can check directly that such a cut-off function satisfies all the requirements.
\end{proof}


\begin{thebibliography}{10}
	
	\bibitem{Alm62}
	Frederick~J. Almgren, Jr.
	\newblock The homotopy groups of the integral cycle groups.
	\newblock {\em Topology}, 1:257--299, 1962.
	
	\bibitem{Alm65}
	Frederick~J. Almgren, Jr.
	\newblock The theory of varifolds.
	\newblock {\em Mimeographed notes, Princeton}, 1965.
	
	\bibitem{ACS17}
	Lucas Ambrozio, Alessandro Carlotto, and Ben Sharp.
	\newblock Compactness analysis for free boundary minimal hypersurfaces.
	\newblock {\em Calc. Var. Partial Differential Equations}, 57(1):57:22, 2018.
	
	\bibitem{CL16}
	Gregory Chambers and Yevgeny Liokumovich.
	\newblock Existence of minimal hypersurfaces in complete manifolds of finite
	volume.
	\newblock {\em arXiv:1609.04058}, 2016.
	
	\bibitem{DeRa18}
	Camillo De~Lellis and Jusuf Ramic.
	\newblock Min-max theory for minimal hypersurfaces with boundary.
	\newblock {\em Ann. Inst. Fourier (Grenoble)}, 68(5):1909--1986, 2018.
	
	\bibitem{DeTa}
	Camillo De~Lellis and Dominik Tasnady.
	\newblock The existence of embedded minimal hypersurfaces.
	\newblock {\em J. Differential Geom.}, 95(3):355--388, 2013.
	
	\bibitem{Fed69}
	Herbert Federer.
	\newblock {\em Geometric measure theory}.
	\newblock Die Grundlehren der mathematischen Wissenschaften, Band 153.
	Springer-Verlag New York Inc., New York, 1969.
	
	\bibitem{Fra00}
	Ailana Fraser.
	\newblock On the free boundary variational problem for minimal disks.
	\newblock {\em Comm. Pure Appl. Math.}, 53(8):931--971, 2000.
	
	\bibitem{FrLi}
	Ailana Fraser and Martin Man-chun Li.
	\newblock Compactness of the space of embedded minimal surfaces with free
	boundary in three-manifolds with nonnegative {R}icci curvature and convex
	boundary.
	\newblock {\em J. Differential Geom.}, 96(2):183--200, 2014.
	
	\bibitem{Gr87}
	Michael Gr{\"u}ter.
	\newblock Optimal regularity for codimension one minimal surfaces with a free
	boundary.
	\newblock {\em Manuscripta Math.}, 58:295--343, 1987.
	
	\bibitem{GWZ18}
	Qiang Guang, Zhichao Wang, and Xin Zhou.
	\newblock Compactness and generic finiteness for free boundary minimal
	hypersurfaces ({I}).
	\newblock {\em arXiv:1803.01509}, 2018.
	
	\bibitem{IMN17}
	Kei Irie, Fernando Marques, and Andr\'e Neves.
	\newblock Density of minimal hypersurfaces for generic metrics.
	\newblock {\em Ann. of Math. (2)}, 187(3):963--972, 2018.
	
	\bibitem{LZ16}
	Martin Li and Xin Zhou.
	\newblock Min-max theory for free boundary minimal hypersurfaces {I}-regularity
	theory.
	\newblock {\em J. Differential Geom., to appear, arXiv:1611.02612}, 2016.
	
	\bibitem{LMN16}
	Yevgeny Liokumovich, Fernando Marques, and Andr\'e Neves.
	\newblock Weyl law for the volume spectrum.
	\newblock {\em Ann. of Math. (2)}, 187(3):933--961, 2018.
	
	\bibitem{MN12}
	Fernando~C. Marques and Andr\'e Neves.
	\newblock Rigidity of min-max minimal spheres in three-manifolds.
	\newblock {\em Duke Math. J.}, 161(14):2725--2752, 2012.
	
	\bibitem{MN1}
	Fernando~C. Marques and Andr{\'e} Neves.
	\newblock Min-max theory and the {W}illmore conjecture.
	\newblock {\em Ann. of Math. (2)}, 179(2):683--782, 2014.
	
	\bibitem{MN16}
	Fernando~C. Marques and Andr\'e Neves.
	\newblock Morse index and multiplicity of min-max minimal hypersurfaces.
	\newblock {\em Camb. J. Math.}, 4(4):463--511, 2016.
	
	\bibitem{MN17}
	Fernando~C. Marques and Andr\'e Neves.
	\newblock Existence of infinitely many minimal hypersurfaces in positive
	{R}icci curvature.
	\newblock {\em Invent. Math.}, 209(2):577--616, 2017.
	
	\bibitem{MN18}
	Fernando~C Marques and Andr{\'e} Neves.
	\newblock Morse index of multiplicity one min-max minimal hypersurfaces.
	\newblock {\em arXiv:1803.04273}, 2018.
	
	\bibitem{MNS17}
	Fernando~C. Marques, Andr\'{e} Neves, and Antoine Song.
	\newblock Equidistribution of minimal hypersurfaces for generic metrics.
	\newblock {\em Invent. Math.}, 216(2):421--443, 2019.
	
	\bibitem{Mon16}
	Rafael Montezuma.
	\newblock Min-max minimal hypersurfaces in non-compact manifolds.
	\newblock {\em J. Differential Geom.}, 103(3):475--519, 2016.
	
	\bibitem{Mor84}
	Frank Morgan.
	\newblock Examples of unoriented area-minimizing surfaces.
	\newblock {\em Trans. Amer. Math. Soc.}, 283(1):225--237, 1984.
	
	\bibitem{Mor86}
	Frank Morgan.
	\newblock A regularity theorem for minimizing hypersurfaces modulo {$\nu$}.
	\newblock {\em Trans. Amer. Math. Soc.}, 297(1):243--253, 1986.
	
	\bibitem{Pi}
	Jon~T. Pitts.
	\newblock {\em Existence and regularity of minimal surfaces on {R}iemannian
		manifolds}, volume~27 of {\em Mathematical Notes}.
	\newblock Princeton University Press, Princeton, N.J.; University of Tokyo
	Press, Tokyo, 1981.
	
	\bibitem{SS}
	R.~Schoen and L.~Simon.
	\newblock Regularity of stable minimal hypersurfaces.
	\newblock {\em Comm. Pure Appl. Math.}, 34(6):741--797, 1981.
	
	\bibitem{Sharp17}
	Ben Sharp.
	\newblock Compactness of minimal hypersurfaces with bounded index.
	\newblock {\em J. Differential Geom.}, 106(2):317--339, 2017.
	
	\bibitem{Si}
	L.~Simon.
	\newblock {\em Lectures on geometric measure theory}, volume~3 of {\em
		Proceedings of the Centre for Mathematical Analysis, Australian National
		University}.
	\newblock Australian National University, Centre for Mathematical Analysis,
	Canberra, 1983.
	
	\bibitem{Song18}
	Antoine Song.
	\newblock Existence of infinitely many minimal hypersurfaces in closed
	manifolds.
	\newblock {\em arXiv:1806.08816}, 2018.
	
	\bibitem{Wang19}
	Zhichao Wang.
	\newblock Compactness and generic finiteness for free boundary minimal
	hypersurfaces ({II}).
	\newblock {\em arXiv:1906.08485}, 2019.
	
	\bibitem{Whi91}
	Brian White.
	\newblock The space of minimal submanifolds for varying {R}iemannian metrics.
	\newblock {\em Indiana Univ. Math. J.}, 40(1):161--200, 1991.
	
	\bibitem{Wh10}
	Brian White.
	\newblock The maximum principle for minimal varieties of arbitrary codimension.
	\newblock {\em Comm. Anal. Geom.}, 18(3):421--432, 2010.
	
	\bibitem{Whi17}
	Brian White.
	\newblock On the bumpy metrics theorem for minimal submanifolds.
	\newblock {\em Amer. J. Math.}, 139(4):1149--1155, 2017.
	
	\bibitem{Yau82}
	Shing~Tung Yau.
	\newblock Problem section.
	\newblock In {\em Seminar on {D}ifferential {G}eometry}, volume 102 of {\em
		Ann. of Math. Stud.}, pages 669--706. Princeton Univ. Press, Princeton, N.J.,
	1982.
	
	\bibitem{Zhou15}
	Xin Zhou.
	\newblock Min-max minimal hypersurface in {$(M^{n+1},g)$} with {$Ric>0$} and
	{$2 \leq n\leq 6$}.
	\newblock {\em J. Differential Geom.}, 100(1):129--160, 2015.
	
	\bibitem{Zhou17}
	Xin Zhou.
	\newblock Min-max hypersurface in manifold of positive {R}icci curvature.
	\newblock {\em J. Differential Geom.}, 105(2):291--343, 2017.
	
	\bibitem{Zhou19}
	Xin Zhou.
	\newblock On the multiplicity one conjecture in min-max theory.
	\newblock {\em arXiv preprint arXiv:1901.01173}, 2019.
	
	\bibitem{ZZ17}
	Xin Zhou and Jonathan Zhu.
	\newblock Min-max theory for cmc hypersurfaces.
	\newblock {\em Invent. Math., to appear, arXiv:1707.08012}, 2017.
	
\end{thebibliography}
\bibliographystyle{plain}
\end{document}